\documentclass[twoside=semi,leqno,abstracton,bibliography=totoc%,DIV=11%classic
]{scrartcl}
\pdfoutput=1
\setkomafont{subtitle}{\itshape}
\setkomafont{disposition}{\normalfont\normalcolor\bfseries}
\setkomafont{dictum}{\scshape}
\setkomafont{paragraph}{\bfseries}
\setkomafont{pagehead}{\normalfont\upshape\small}
\setkomafont{pagenumber}{\normalfont\upshape\normalsize}
\usepackage[arxiv%
]{MyKoma}
\usepackage{dsfont}

\makeatletter
\def\dom{\mathrm{dom}}

\DeclareMathOperator*{\esssup}{ess\,sup}

\makeatother

\begin{document}

% \titlehead{aaa}
\title{Maximum Lebesgue Extension of Monotone Convex Functions}
%\titlehead{CARF-F-214}
%\dedication{Dedicated to bbb}
%\subject{aaa}
% \subtitle{Non-Dominated Case}%
\runtitle{Maximum Lebesgue Extension}
% \subject{bbbb}
%\subtitle{Author's Response and Correction Requests}
\author{Keita Owari}%
\runauthor{K. Owari}

%\myThanks[1]{To appear in \emph{Math. Financ. Econ.}, DOI : 10.1007/s11579-013-0111-z
  % The author gratefully acknowledges the financial support from the
  % Center for Advanced Research in Finance (CARF) at the
  % Graduate School of Economics of the University of Tokyo
%}
 
\address{Graduate School of Economics, The University of Tokyo\newline
  7-3-1 Hongo, Bunkyo-ku, Tokyo 113-0033, Japan}

\email{owari@e.u-tokyo.ac.jp}
% \ArXiV{112233} 
% \keyAMS{46E30, 47H07, 46N10, 91G80, 91B30 }%
% \keyJEL{C02, C60}%

\keyWords{Monotone Convex Functions, Lebesgue Property,
  Order-Continuity, Order-Continuous Banach Lattices, Uniform
  Integrability, Convex Risk Measures}

% \date{\today}
% %%% History %%%
\FrstVer{30 Apr. 2013, Accepted: 8 Jan.
  2014}% \Accepted{16 November 2013}

\ToAppear{Journal of Functional Analysis}
\DOI{10.1016/j.jfa.2014.01.002}
%\CurrVer{\today}

%%%% ABSTRACT %%%%%

\abstract{%
  Given a monotone convex function on the space of essentially bounded
  random variables with the Lebesgue property (order continuity), we
  consider its extension preserving the Lebesgue property to as big
  solid vector space of random variables as possible. We show that
  there exists a maximum such extension, with explicit construction,
  where the maximum domain of extension is obtained as a (possibly
  proper) subspace of a natural Orlicz-type space, characterized by a
  certain uniform integrability property. As an application, we
  provide a characterization of the Lebesgue property of monotone
  convex function on arbitrary solid spaces of random variables in
  terms of uniform integrability and a ``nice'' dual representation of
  the function.
}
\maketitle

\section{Introduction}
\label{sec:Intro}

%$\mathbb1$
%$X\mathbbold1$ $\mathbbm1$ $\mathds{1}$ $\mathbf1$

Motivated by the study of convex risk measures in financial
mathematics, we address a ``regular'' extension problem of monotone
convex functions.  Let $L^0$ be the space of all finite random
variables (measurable functions) on a given probability space
$(\Omega,\FC,\PB)$ modulo $\PB$-almost sure (a.s.)  equality, and we
say that a linear subspace $\Xs\subset L^0$ is solid if $X\in \Xs$ and
$|Y|\leq |X|$ a.s. imply $Y\in \Xs$. By a monotone convex function on
a solid space $\Xs\subset L^0$, we mean a convex function
$\varphi:\Xs\rightarrow (-\infty,\infty]$ which is monotone increasing
w.r.t. the a.s. pointwise order.

We are interested in monotone convex functions on some solid space
$\Xs$ having the following regularity property called the
\emph{Lebesgue property}: for any sequence $(X_n)_n\subset \Xs$,
\begin{equation}
  \label{eq:LebesgueIntro1}
  \exists Y\in \Xs,\,|X_n|\leq Y\, (\forall n)\text{  and }X_n\rightarrow X\in \Xs\text{ a.s. }
  \Rightarrow \varphi(X)=\lim_n\varphi(X_n).
\end{equation}
Note that all $L^p$ spaces are solid, and when
$\Xs=L^1:=L^1(\Omega,\FC,\PB)$ and $\varphi(X)=\EB[X]$, this is
nothing but the dominated convergence theorem. When $\Xs=L^\infty$,
(\ref{eq:LebesgueIntro1}) reduces to
\begin{equation}
  \label{eq:LebesgueIntroLinfty}
  \sup_n\|X_n\|_\infty<\infty 
  \text{ and }X_n\rightarrow X\text{ a.s. }
  \Rightarrow\, \varphi(X)=\lim_n \varphi(X_n),
\end{equation}
and a number of practically important monotone convex functions on
$L^\infty$ satisfy this. 

Now given a monotone convex function $\varphi_0$ on $L^\infty$ with
the Lebesgue property (\ref{eq:LebesgueIntroLinfty}), we consider its
extension to some big solid space \emph{preserving the Lebesgue
  property} in the form of (\ref{eq:LebesgueIntro1}) (such extensions
do make sense). Of course there may be several such extensions, but we
are interested in the maximum one. So the central question of the
paper is:
\begin{question}
  \label{ques:1}
  Given a monotone convex function $\varphi_0$ on $L^\infty$ with the
  Lebesgue property (\ref{eq:LebesgueIntroLinfty}), does there exist a
  maximum extension preserving the Lebesgue property in the sense of
  (\ref{eq:LebesgueIntro1})? i.e., is there a pair
  $(\hat\varphi,\widehat\Xs)$ of a solid space $\widehat\Xs\subset L^0$ and
  a monotone convex function $\hat\varphi$ with the Lebesgue property
  on $\widehat\Xs$ such that $\hat\varphi|_{L^\infty}=\varphi_0$ and for
  any such pair $(\varphi,\Xs)$, one has $\Xs\subset \widehat\Xs$ and
  $\varphi=\hat\varphi|_\Xs$?
\end{question}

As a first (trivial) example, we briefly see what happens when
$\varphi_0$ is linear.
\begin{example}
  \label{ex:PosLin}
  Let $\varphi_0$ be a \emph{positive} (monotone) linear functional on
  $L^\infty$.  Then it is finite-valued and identified with a
  \emph{finitely additive} measure $\nu_0(A):=\varphi_0(\mathds1_A)$
  as $\varphi_0(X)=\int_\Omega Xd\nu_0$, while
  (\ref{eq:LebesgueIntroLinfty}) is equivalent to saying that $\nu_0$
  is $\sigma$-additive. If the latter is the case, the ``usual''
  integral $\hat\varphi(X):=\int_\Omega Xd\nu_0$ defines a
  Lebesgue-preserving extension of $\varphi_0$ to
  $\LC^1(\nu_0):=\{X\in L^0:\, \int_\Omega|X|d\nu_0<\infty\}$. On the
  other hand, if $\varphi$ is a monotone convex function on a solid
  space $\Xs\subset L^0$ with (\ref{eq:LebesgueIntro1}) and
  $\varphi|_{L^\infty}=\varphi_0$, it is easy that $\varphi$ must be
  positive, linear and finite on $\Xs$.  Then $\int
  |X|d\nu_0=\lim_n\hat\varphi(|X|\wedge n)=\lim_n\varphi_0(|X|\wedge
  n)=\lim_n\varphi(|X|\wedge n)=\varphi(|X|)<\infty$ if $X\in\Xs$,
  hence $\Xs\subset \LC^1(\nu_0)$, where the first equality follows
  from the monotone convergence theorem, and the fourth from the
  Lebesgue property of $\varphi$ on $\Xs$. Similarly, but with
  $X\mathds1_{\{|X|\leq n\}}$ instead of $|X|\wedge n$, we see also
  that $\varphi=\hat\varphi|_\Xs$. Namely,
  $(\hat\varphi,\LC^1(\nu_0))$ is the maximum Lebesgue-preserving
  extension of $\varphi_0$.
\end{example}

This is just an exercise of measure theory, and we see that
Question~\ref{ques:1} is well-posed at least when $\varphi_0$ is
linear. Slight surprisingly, the main result
(Theorem~\ref{thm:MainLebExt1}) of this paper states that the answer
to Question~\ref{ques:1} is YES as long as the original function
$\varphi_0$ is \emph{finite everywhere} on $L^\infty$ (this is
automatic when $\varphi$ is linear by definition). Moreover, the
maximum extension $(\hat\varphi,\widehat\Xs)$ is \emph{explicitly
  constructed}.

We first construct a candidate of $\hat\varphi$ in a rather ad-hoc way
on a certain convex cone of $L^0$ containing $L^\infty$ and the
positive cone $L^0_+$. Then based on a simple observation
(Lemma~\ref{lem:Observation2}), we introduce an \emph{Orlicz-type
  space} associated to $\hat\varphi$, that we denote by
$M^{\hat\varphi}_u$, beyond which Lebesgue-preserving extension is not
possible. After checking that the candidate $\hat\varphi$ is
well-defined on this space as a \emph{finite} monotone convex
function, we finally verify that the space $M^{\hat\varphi}_u$ can be
made into an \emph{order-continuous} Banach lattice with respect to a
natural gauge norm (Theorem~\ref{thm:MphiUOrderContiBanach}) with a
suitable change of measure, which together with an extended
Namioka-Klee theorem by \citep{MR2648595} eventually yields that
$\hat\varphi$ is Lebesgue on $M^{\hat\varphi}_u$ and the pair
$(\hat\varphi,M^{\hat\varphi}_u)$ is the desired maximum extension.
The space $M^{\hat\varphi}_u$ is, as the notation suggests, a subspace
of the ``Orlicz heart'' $M^{\hat\varphi}$ of $\hat\varphi$, and the
subscript ``$u$'' stands for the ``uniform integrability'' that
characterizes the elements of $M^{\hat\varphi}_u$. This point will be
made clear in Theorem~\ref{thm:MuCharact}.

As an application, we provide a characterization of the Lebesgue
property of finite monotone convex functions $\psi$ on an arbitrary
solid space of random variables of the form \emph{Fatou property plus
  ``something extra''}, with the ``extra'' being either a certain
``uniform integrability'' or a ``good'' dual representation of $\psi$,
both of which are stated using the conjugate of $\psi|_{L^\infty}$
(Theorem~\ref{thm:JSTGeneral}). This generalizes a result known as the
\emph{Jouini-Schachermayer-Touzi} theorem
\citep{jouini06:_law_fatou}. There the comparison of a function $\psi$
on a solid space $\Xs$ and the maximum Lebesgue-preserving extension
of the restriction $\psi|_{L^\infty}$ plays a key role.

\subsection{A Motivation from Financial Mathematics: Convex Risk Measures}
\label{sec:Question}

An initial motivation of this work was to provide an ``efficient'' way
to the study of convex risk measures for unbounded risks. In
mathematical finance, a \emph{convex risk measure} on a solid space
$\Xs\subset L^0$ is---up to a change of sign---a monotone convex
function $\rho$ on $\Xs$ such that $\rho(X+c)=\rho(X)+c$ whenever $c$
is a constant (cash-invariance).  This notion was introduced by
\citep{MR1850791,MR1932379, FrittelliRosazza20021473} as a possible
replacement of \emph{Value at Risk}. See \citep[][Ch.~4]{MR2779313}
for the background of this notion.  Since then, convex risk measures
on $L^\infty$ (i.e.  for bounded risks) have been extensively studied,
establishing a number of their fine properties as well as examples
\citep[see e.g.][]{delbaen12:_monet_utilit_funct,MR2779313}. However,
$L^\infty$ is clearly too small to capture the actual risks, and a key
current direction is the analysis of risk measures \emph{beyond
  bounded risks}.  A natural way is to pick up a particular space, and
then to reconstruct a whole theory with careful analysis of the
structure of the new space, e.g., $L^p$
\citep{MR3035989,MR2507760}, Orlicz spaces/hearts
\citep{MR2509268, orihuela_ruiz12:_lebes_orlic,MR2659479,MR2756015},
abstract locally convex Fréchet lattices \citep{MR2648595}, and $L^0$
\citep{MR2823052} to mention a few.

On the other hand, it seems more efficient to \emph{extend} a convex
risk measure originally defined on $L^\infty$ to some big space, and a
most natural candidate seems the one preserving the Lebesgue
property. Note first that the Lebesgue property of the original risk
measure on $L^\infty$ is reasonable, since (modulo some technicality)
it is necessary to have a finite valued extension to some solid space
properly containing $L^\infty$ (\citep[Theorem~10]{MR2509290}; see the
paragraph after Theorem~\ref{thm:JSTLinfty} for detail).  Next, the
Lebesgue property implies or is equivalent to some other important
properties in application: existence of \emph{$\sigma$-additive
  subgradient}, the inf-compactness of the conjugate, the continuity
for the Mackey topology induced by the good dual space and so on
(\citep{jouini06:_law_fatou}, \citep{MR2648597} and comments after
Theorem~\ref{thm:JSTGeneral} for precise information).  Also,
functions with the Lebesgue property are stable for the practically
common procedure of approximating unbounded random variables by
suitable ``truncation'', and a ``nearly'' converse implication is also
true (Remark~\ref{rem:Truncation}).  This is computationally useful,
and it also means roughly that an extension preserving the Lebesgue
property retains the basic structure of the original function to the
extended domain.

Several other types of extensions may be possible of course, and some
of those have already appeared in literature (see
Section~\ref{sec:Other}). Especially, \citep{MR2943186} considered an
extension preserving the \emph{Fatou property} (order lower
semicontinuity), proving that any \emph{law-invariant} convex risk
measure with the Fatou property on $L^\infty$ is uniquely extended to
$L^1$ preserving the Fatou property.  In contrast, a simple example
shows that Lebesgue-preserving extension to $L^1$ or to some
``common'' reasonable space is not possible even if the original
function is law-invariant (see Example~\ref{ex:Modular} and discussion
that precedes).  Thus it is worthwhile to ask how far a convex risk
measure originally defined on $L^\infty$ with the Lebesgue property
can be extended preserving the Lebesgue property, or more intuitively,
how far a ``good'' risk measure can remain ``good''.  In
Section~\ref{sec:ConvRiskFunc}, we shall examine our main results in
the context of convex risk measures with some concrete examples.

\section{Preliminaries}
\label{sec:MonConvSolid}

We use the probabilistic notation. Let $(\Omega,\FC,\PB)$ be a
probability space which will be fixed throughout, and
$L^0:=L^0(\Omega,\FC,\PB)$ denotes the space of all equivalence
classes of measurable functions (or random variables) over
$(\Omega,\FC,\PB)$ modulo $\PB$-almost sure (a.s.) equality. As usual,
we do not distinguish an element of $L^0$ and its representatives, and
inequalities between (classes of) measurable functions are to be
understood in the a.s. sense, i.e., $X\leq Y$ a.s. which means more
precisely that $f\leq g$ a.s. for any representatives $f$ and $g$ of
$X$ and $Y$, respectively. This a.s. pointwise inequality defines a
partial order on $L^0$ by which $L^0$ is an \emph{order-complete}
Riesz space (vector lattice) with the \emph{countable-sup property}.
By a \emph{solid space} $\Xs$, we mean, in this paper, a \emph{solid
  vector subspace} (ideal) $\Xs$ of $L^0$, i.e., a vector subspace of
$L^0$ such that $|X|\leq |Y|$ and $Y\in \Xs$ imply $X\in \Xs$
(solid). Note that any such $\Xs$ is an order complete Riesz space
with the countable sup-property on its own right, and $\Xs$ contains
$L^\infty:=L^\infty(\Omega,\FC,\PB)$ as soon as it contains the
constants. We denote $\Xs_+:=\{X\in \Xs:\, X\geq 0\}$ (the positive
cone). Finally, we write $\EB[X]=\int_\Omega X(\omega)\PB(d\omega)$
(expectation w.r.t. $\PB$) for $X\in L^0$ as long as the integral
makes sense, and $\EB_Q[X]=\int_\Omega X(\omega)Q(d\omega)$ for other
probability measures $Q\ll\PB$.

By a \emph{monotone} convex function on a solid space $\Xs\subset
L^0$, we mean a proper convex function
$\varphi:\Xs\rightarrow(-\infty,\infty]$ which is monotone
\emph{increasing} in the a.s. order:
\begin{equation}
  \label{eq:Monotone}
  \forall X,Y\in\Xs,\,   X\leq Y \,(\text{a.s.})\,\Rightarrow\, \varphi(X)\leq \varphi(Y).
\end{equation}

\begin{definition}
  \label{dfn:Regularity}
  For a monotone convex function $\varphi$ on a solid space
  $\Xs\subset L^0$, we say that
\begin{enumerate}
\item $\varphi$ satisfies the \emph{Fatou property} (or simply
  $\varphi$ is Fatou) if for any $(X_n)_n\subset \Xs$,
  \begin{equation}
    \label{eq:FatouX2}
    \exists Y\in \Xs_+\text{ such that }
    |X_n|\leq Y,\,\forall n\text{ and }X_n\rightarrow X\text{ a.s. }
    \Rightarrow\,\varphi(X)\leq\liminf_n\varphi(X_n).
  \end{equation}
\item $\varphi$ satisfies the \emph{Lebesgue property} (or $\varphi$
  is Lebesgue) if for any $(X_n)_n\subset \Xs$,
  \begin{equation}
    \label{eq:LebX2}
    \exists Y\in \Xs_+\text{ such that } |X_n|\leq Y,\,\forall n\text{ and }X_n\rightarrow X\text{ a.s. }\Rightarrow\,\varphi(X)=\lim_n\varphi(X_n).
  \end{equation}
\end{enumerate}
\end{definition}

\begin{remark}[Lebesgue property and order-continuity]
  \label{rem:Lebesgue-OrderConti}
  By the countable-sup property of $\Xs$ (as a solid vector subspace
  (ideal) of $L^0$), the Lebesgue property (\ref{eq:LebX2}) is
  equivalent to the generally stronger \emph{order continuity}:
  $\varphi(X_\alpha)\rightarrow \varphi(X)$ if a net $X_\alpha$
  converges \emph{in order} to $X$ ($X_\alpha\stackrel{o}\rightarrow
  X$), i.e., if there exists a decreasing net $(Y_\alpha)_\alpha
  \subset \Xs$ (with the same index set) such that $|X-X_\alpha|\leq
  Y_\alpha\downarrow 0$ (in the lattice sense).  Indeed, for a
  \emph{sequence} (or slightly more generally a \emph{countable net})
  $(X_n)_n\subset \Xs$, the order convergence
  $X_n\stackrel{o}\rightarrow X$ is equivalent to the \emph{dominated
    a.s. convergence}: $|X_n|\leq Y$ ($\forall n$) for some
  $Y\in\Xs_+$ and $X_n\rightarrow X$ a.s., thus the Lebesgue property
  (\ref{eq:LebX2}) is nothing but the \emph{$\sigma$-order
    continuity}. On the other hand, for monotone (increasing)
  functions, the order continuity is equivalent to the continuity from
  above: $X_\alpha\downarrow X$ $\Rightarrow$
  $\varphi(X_\alpha)\downarrow \varphi(X)$, and by the countable-sup
  property, any such decreasing net admits a sequence
  $(X_{\alpha_n})_n \subset (X_\alpha)_\alpha$ such that
  $X_{\alpha_n}\downarrow X$. Consequently, the $\sigma$-order
  continuity implies $\varphi(X)\leq
  \lim_\alpha\varphi(X_\alpha)=\inf_\alpha\varphi(X_\alpha)\leq
  \inf_n\varphi(X_{\alpha_n})=\varphi(X)$.  A similar remark applies
  also to the Fatou property (\ref{eq:FatouX2}) and the
  order-\emph{lower semicontinuity}. For further information, see
  e.g. \citep[Ch.~8, 9]{aliprantis_border06}.
\end{remark}

The Lebesgue and Fatou properties are more ``universal'' than the
corresponding topological regularities as long as we discuss functions
of random variables, in the sense that they are comparable between
different spaces. In fact, it is clear from the definition that if
$\Xs$ and $\Ys$ are solid spaces with $\Xs\subset \Ys(\subset L^0)$
and if a function $\varphi$ on $\Ys$ has the Lebesgue property, then
the restriction $\varphi|_\Xs$ automatically has the Lebesgue property
on $\Xs$, and the same is true for the Fatou property. In particular,
the class of monotone convex functions with the Lebesgue property on
solid spaces $(\varphi,\Xs)$ is partially ordered simply by
$(\varphi,\Xs)\preceq (\psi,\Ys)$ iff $\Xs\subset \Ys$ and
$\varphi=\psi|_\Xs$, and the maximum extension preserving the Lebesgue
property does make sense, while, for instance, maximum extension of
norm-continuous function on $L^\infty$ preserving the topological
continuity does not much make sense:

\begin{definition}[Lebesgue Extension]\label{dfn:LebExt}
  Let $\Xs_0\subset L^0$ be a solid space and $\varphi_0:\Xs_0
  \rightarrow (-\infty,\infty]$ a monotone convex function with the
  Lebesgue property (\ref{eq:LebX2}) on $\Xs_0$. Then we say that
  $(\varphi,\Xs)$ is a \emph{Lebesgue extension} of
  $(\varphi_0,\Xs_0)$ if $\Xs\subset L^0$ is a solid space containing
  $\Xs_0$, $\varphi:\Xs\rightarrow (-\infty,\infty]$ is a monotone
  convex function with the Lebesgue property on $\Xs$ and
  $\varphi_0=\varphi|_\Xs$. If there exists a Lebesgue extension
  $(\hat\varphi,\widehat\Xs)$ such that $\Xs\subset \widehat\Xs$ and
  $\varphi=\hat\varphi|_\Xs$ for any Lebesgue extension
  $(\varphi,\Xs)$ of $(\varphi_0,\Xs_0)$, then we say that
  $(\hat\varphi, \widehat\Xs)$ is the \emph{maximum Lebesgue
    extension} of $(\varphi_0,\Xs_0)$.
\end{definition}
If there is no risk of confusion, we omit $\Xs_0$ and simply say
e.g. $(\varphi,\Xs)$ is a Lebesgue extension of $\varphi_0$. In fact,
we shall be discussing in the sequel the Lebesgue extensions of a
monotone convex function $\varphi_0$ on $L^\infty$, i.e., always
$\Xs_0=L^\infty$.

\subsection{Monotone Convex Functions on $L^\infty$}
\label{sec:MonConvLinfty}

Here we briefly summarize some basic facts on the monotone convex
functions on $L^\infty$.  Note first that the Fatou and Lebesgue
properties (\ref{eq:LebX2}) and (\ref{eq:FatouX2}), respectively, for
a proper convex function $\varphi$ on $L^\infty$ are equivalently
stated as
\begin{align}
  \tag{\text{\ref{eq:FatouX2}${}_\infty$}}
  \label{eq:FatouLinfty1}
  &\sup_n\|X_n\|_\infty<\infty\text{ and }X_n\rightarrow X\text{
    a.s. }\Rightarrow\varphi(X)\leq \liminf_n\varphi(X_n),\\
  \tag{\text{\ref{eq:LebX2}${}_\infty$}}%{$\text{\ref{eq:LebX2}}_\infty$}}
  \label{eq:LebLinfty1}
  &\sup_n\|X_n\|_\infty<\infty\text{ and }X_n\rightarrow X\text{
    a.s. }\Rightarrow\varphi(X)=\lim_n\varphi(X_n),
\end{align}
while (\ref{eq:FatouLinfty1}) is equivalent to the lower
semicontinuity w.r.t. $\sigma(L^\infty,L^1)$ (the weak*
topology). Indeed, a \emph{convex} set $C\subset L^\infty$ is
$\sigma(L^\infty,L^1)$-closed if and only if for every $c>0$,
$C\cap\{X:\,\|X\|_\infty\leq c\}$ is closed in $L^0$ which is a
well-known consequence of the Krein-Šmulian theorem (see
e.g. \citep{MR0372565}). Thus by Fenchel-Moreau theorem, the Fatou
property of a proper convex function $\varphi$ on $L^\infty$ is
equivalent to the dual representation
\begin{equation}
  \label{eq:SigmaAddDual1}
  \varphi(X)=\sup_{Z\in L^1}(\EB[XZ]-\varphi^*(Z))
\end{equation}
where $\varphi^*$ is the Fenchel-Legendre transform (conjugate) of
$\varphi$ in $\langle L^\infty,L^1\rangle$ duality:
\begin{equation}
  \label{eq:ConjugateLinfty1}
  \varphi^*(Z):=\sup_{X\in L^\infty}(\EB[XZ]-\varphi(X)),\quad \forall Z\in L^1,
\end{equation}
Then the monotonicity of $\varphi$ is equivalent to
$\dom\varphi^*\subset L^1_+$, i.e.,
\begin{equation}
  \label{eq:MonotoneZ}
  Z\in L^1,\,\varphi^*(Z)<\infty\,\Rightarrow\, Z\geq 0.
\end{equation}

The next characterization of the Lebesgue property
(\ref{eq:LebLinfty1}) is a ramification of a result known as the
\emph{Jouini-Schachermayer-Touzi theorem} (JST in short) in financial
mathematics. In the case of convex risk measure (up to change of sign,
i.e. $\varphi(X+c)=\varphi(X)+c$ if $c\in\RB$), it was first obtained
by \citep{jouini06:_law_fatou} with an additional separability
assumption, and the latter assumption was removed later by
\citep{MR2648597} using a homogenization trick. See also
\citep{orihuelaRuizGalan12:_james,orihuela_ruiz12:_lebes_orlic}.
\begin{theorem}[cf. \citep{jouini06:_law_fatou, MR2648597,
    orihuelaRuizGalan12:_james, orihuela_ruiz12:_lebes_orlic} for
  convex risk measures]
  \label{thm:JSTLinfty}
  For a finite monotone convex function
  $\varphi:L^\infty\rightarrow\RB$ satisfying the Fatou property
  (\ref{eq:FatouLinfty1}), the following are equivalent:
  \begin{enumerate}
  \item $\varphi$ has the Lebesgue property (\ref{eq:LebLinfty1});
  \item $\{Z\in L^1:\, \varphi^*(Z)\leq c\}$ is weakly compact in
    $L^1$ for each $c>0$;
  \item for each $X\in L^\infty$, the supremum $\sup_{Z\in
      L^1}(\EB[XZ]-\varphi^*(Z))$ is attained;
  \item $\varphi$ is continuous for the Mackey topology
    $\tau(L^\infty,L^1)$.
  \end{enumerate}

\end{theorem}

\begin{proof}
  (1) $\Leftrightarrow$ (2) $\Rightarrow$ (3) can be proved in the
  same way as \citep{jouini06:_law_fatou}, while given the finiteness
  and $\sigma(L^\infty,L^1)$-lower semicontinuity of $\varphi$, (2)
  $\Leftrightarrow$ (4) is also a well-known fact in convex analysis
  (e.g. \citep[Propositions~1 and 2]{MR0160093}).  For (3)
  $\Rightarrow$ (2), observe that for each $Z\in L^1$ and $\alpha>0$,
  $\varphi^*(Z)\geq \EB[\alpha\mathrm{sgn}(Z)Z]
  -\varphi(\alpha\mathrm{sgn}(Z)) \geq\alpha\|Z\|_1-\varphi(-\alpha)$
  where $\mathrm{sgn}(Z):=\mathds1_{\{Z>0\}}-\mathds1_{\{Z<0\}} \in
  L^\infty$. Since $\varphi$ is finite-valued, this shows that
  $\lim_{\|Z\|_1\rightarrow \infty}\varphi^*(Z)/\|Z\|_1=\infty$ (i.e.,
  $\varphi^*$ is coercive). Then the implication (3) $\Rightarrow$ (2)
  follows from \emph{coercive James's theorem} due to
  \citep{orihuelaRuizGalan12:_james} (recalled below as
  Theorem~\ref{thm:James}).
\end{proof}

Finally, we note that the Lebesgue property on $L^\infty$ is
reasonable. In fact, when $(\Omega,\FC,\PB)$ is \emph{atomless} (which
is not a restriction in practice), a sufficient condition for the
Lebesgue property (\ref{eq:LebLinfty1}) on $L^\infty$ for monotone
convex function $\varphi$ is that it has a \emph{finite-valued}
extension to a solid space $\Xs\supsetneq L^\infty$ such that
$X\in\Xs$ and $\mathrm{law}(Y)=\mathrm{law}(X)$ $\Rightarrow$
$Y\in\Xs$ (\emph{rearrangement invariant}). See
\citep[Th.~3]{MR2509290} where this is proved for convex risk
measures, and an almost same proof still works for general
\emph{finite} monotone convex functions. All $L^p$ ($0\leq p\leq
\infty$), Orlicz spaces and Orlicz hearts (the Morse subspaces of the
corresponding Orlicz spaces) are of this type.  Thus functions
$\varphi$ that violate this condition are rarely of practical
interest.

\subsection{Other extensions and general remarks}
\label{sec:Other}

We emphasize that the preservation of the Lebesgue property is
crucial. In fact, \emph{any finite} monotone convex function on
$L^\infty$ has \emph{an} extension to the whole $L^0$ if one does not
mind any regularity or uniqueness. Indeed, let
\begin{equation}
  \label{eq:CDK1}
  \varphi_{\mathrm{ext}}(X):=\lim_n\lim_m\varphi_0(( X\vee (-n)\wedge m),\quad X\in L^0.
\end{equation}
Noting that $(X\vee (-n))\wedge m=X$ if $\|X\|_\infty\leq m,n<\infty$,
this is well-defined on $L^0$ with values in $[-\infty,\infty]$, and
$\varphi_{\mathrm{ext}}|_{L^\infty}=\varphi_0$. But it is not a
regular nor unique extension in any reasonable sense, or it may even
be improper.  In the context of convex risk measures,
\citep{MR2211713} studied this type extension, providing a necessary
and sufficient condition for $\varphi_{\mathrm{ext}}$ to avoid the
value $-\infty$ (hence proper), but even in that case, we have no
regularity nor uniqueness.

\begin{remark}\label{rem:Truncation}
  In application, one often hopes to approximate \emph{unbounded}
  $X\in L^0$ by \emph{bounded} ones via suitable \emph{truncation} as
  $X\mathds1_{\{|X|\leq n\}}\stackrel{n}\rightarrow X$, $(X\vee
  (-m))\wedge n\stackrel{n,m}\rightarrow X$. As these convergences are
  \emph{order convergences}, Remark~\ref{rem:Lebesgue-OrderConti}
  tells us that monotone convex functions $\varphi$ with the Lebesgue
  property are stable for this sort of approximations:
  \begin{equation}
    \label{eq:LebesgueInfSup}
    \varphi(X)=\lim_{m\rightarrow\infty}\lim_{n\rightarrow\infty}
    \varphi((X\vee(-m))\wedge n)% \\
    =\lim_{n\rightarrow\infty}
    \varphi(X\mathds1_{\{|X|\leq n\}}),
  \end{equation}
  and two limits in the middle expression are interchangeable. In
  fact, a sort of converse is also true: a finite monotone convex
  function $\varphi$ with the Fatou property on a solid space
  $\Xs\subset L^0$ has the Lebesgue property if and only if for any
  \emph{countable} net $(X_\alpha)_\alpha$,
  \begin{equation}
    \label{eq:AproxBDD}
    X_\alpha\in L^\infty,\, |X_\alpha|\leq |X|,\,
    \forall \alpha,\text{ and }\,X_\alpha\rightarrow X\text{ a.s. }
    \Rightarrow \, \varphi(X_\alpha)\rightarrow\varphi(X).
  \end{equation}
  See Proposition~\ref{prop:Truncation}. In particular, the maximum
  Lebesgue extension tells us the precise extent to which \emph{any
    ``reasonable'' truncation procedures} safely work.

\end{remark}

A closely related question, recently addressed by \citep{MR2943186},
is the extension preserving the Fatou property (instead of Lebesgue).
There the ``$L^1$-closure'' of $\varphi_0$ given by
$\bar\varphi^1_0(X):=\sup_{Y\in L^\infty}(\EB[XY]-\varphi_0^*(Y))$ on
$L^1$ is considered. This is clearly proper and (weakly) lower
semicontinuous (hence Fatou) on $L^1$ as soon as $\dom\varphi_0^*\cap
L^\infty\neq \emptyset$, while it is not clear if $\bar\varphi^1$ is
an extension of $\varphi_0$, i.e., if
$\bar\varphi^1|_{L^\infty}=\varphi_0$.  \citep[Theorem~2.2]{MR2943186}
proved that this is the case if $\varphi_0$ is \emph{law-invariant}
(i.e. $X\stackrel{\text{law}}=Y$ $\Rightarrow$
$\varphi_0(X)=\varphi_0(Y)$), and then $\bar\varphi^1$ is the
\emph{unique} lower semi-continuous extension of $\varphi$ to
$L^1$. In particular, every \emph{law-invariant} convex risk measure
has a ``Fatou'' extension to $L^1$. In contrast, the Lebesgue property
may not be preserved to $L^1$ (even if law-invariant) as the next
example illustrates.

\begin{example}[Modular]
  \label{ex:Modular}
  Let $\Phi:\RB\rightarrow\RB_+$ be a lower semicontinuous even convex
  function with $\Phi(0)=0$, and
  $\lim_{x\rightarrow\infty}\Phi(x)=\infty$ (i.e., a \emph{finite
    Young function}). Then put
  \begin{equation}
    \label{eq:Modular1}
    \rho_\Phi(X):=\EB[\Phi(X^+)]=\EB[\Phi(X\vee 0)],\,X\in L^0.
  \end{equation}
  This is clearly a law-invariant $[0,\infty]$-valued monotone convex
  function with $\rho_\Phi(0)=0$ satisfying the Fatou property on the
  whole $L^0$ (by Fatou's lemma since $\Phi\geq 0$). Let
  \begin{align}
    L^\Phi&:=\{X\in L^0:\,\exists\alpha>0,\, \EB[\Phi(\alpha|X|)]<\infty\} \quad (\text{Orlicz space}),\\
    M^\Phi&:=\{X\in L^0:\,\forall \alpha>0,\, \EB[\Phi(\alpha|X|)]
    <\infty\}\quad (\text{Orlicz heart}).
  \end{align}
  It always holds $L^\infty\subset M^\Phi\subset L^\Phi\subset L^1$
  and $M^\Phi=L^\Phi$ if $\Phi$ satisfies the so-called
  $\Delta_2$-condition, while if for example $\Phi(x)=e^{|x|}-1$ and
  $(\Omega,\FC,\PB)$ is atomless, then $L^\infty\subsetneq
  M^\Phi\subsetneq L^\Phi\subsetneq L^1$. The function $\rho_\Phi$ is
  Lebesgue on $M^{\Phi}$ since $|X_n|\leq |Y|$ with $Y\in M^\Phi$ and
  $X_n\rightarrow X$ a.s. imply $|\Phi(X_n^+)|\leq \Phi(|Y|)\in L^1$,
  hence $\rho_\Phi(X_n)=\EB[\Phi(X_n^+)]\rightarrow
  \EB[\Phi(X^+)]=\rho_\Phi(X)$ by dominated convergence. On the other
  hand, $\rho_\Phi$ is \emph{not Lebesgue} on $L^\Phi$ unless
  $M^\Phi=L^\Phi$. Indeed, if $X\in L^\Phi\setminus M^\Phi$, and
  $\alpha>0$ is such that $\EB[\Phi(\alpha|X|)]=\infty$, then
  $\rho_\Phi(\alpha|X|\mathds1_{\{|X|>n\}})=\EB[\Phi(\alpha|X|)\mathds1_{\{|X|>n\}}]\equiv
  \infty$ for all $n$ while $0\leq \alpha|X|\mathds1_{\{|X|>n\}}\leq
  \alpha|X|$ and $\alpha|X|\mathds1_{\{|X|>n\}}\rightarrow 0$ a.s. By
  the law-invariance and \citep{MR2943186},
  $(\rho_\Phi,L^1)$ is the unique Fatou-preserving extension of
  $(\rho_\Phi|_{L^\infty},L^\infty)$ which is not Lebesgue on
  $L^\Phi\subsetneq L^1$. Consequently, $\rho_\Phi|_{L^\infty}$ has no
  Lebesgue extension to $L^1$.
  
\end{example}

\section{Statements of Main Results}
\label{sec:MaxLebExt}

We begin with a couple of elementary observations. Let
$\varphi_0:L^\infty\rightarrow\RB$ be a finite monotone convex
function with the Fatou property (\ref{eq:FatouLinfty1}) hence
represented as (\ref{eq:SigmaAddDual1}) by the conjugate
$\varphi^*_0(Z)=\sup_{X\in L^\infty}(\EB[XZ]-\varphi_0(X))$ ($Z\in
L^1$). Let
\begin{align}
  \label{eq:DX}
  \DC_0&:=\left\{X\in L^0: \,X^-Z\in L^1,\,\forall Z\in \dom\varphi^*_0\right\}.
\end{align}
This is not a vector space, but a convex cone containing $L^\infty\cup
L^0_+$, which is \emph{upward solid} in the sense that $X\in\DC_0$ and
$X\leq Y$, then $Y\in \DC_0$ since then $Y^-\leq X^-$.  We then define
\begin{equation}
  \label{eq:phihat1}
  \hat\varphi(X):=\sup_{Z\in \dom\varphi^*_0}\left(\EB[XZ]-\varphi^*_0(Z)\right),\quad \forall X\in\DC_0,
\end{equation}
where $\dom\varphi^*_0:=\{Z\in L^1:\, \varphi^*_0(Z)<\infty\}\subset
L^1_+$ (by (\ref{eq:MonotoneZ})). This is well-defined with values in
$(-\infty,\infty]$ and is \emph{continuous from below}:
\begin{lemma}
  \label{lem:HatPhiMon1}
  Let $\varphi_0$ be a finite monotone convex function with the Fatou
  property on $L^\infty$. Then $\hat\varphi$ defined by
  (\ref{eq:phihat1}) is a proper monotone convex function on $\DC_0$
  with $\hat\varphi|_{L^\infty}=\varphi_0$ and
\begin{equation}
  \label{eq:HatPhiMon1}
  X_n\in\DC_0,\,X_n\uparrow X\in L^0\text{ a.s. }\Rightarrow\, \hat\varphi(X)=\lim_n\hat\varphi(X_n).
\end{equation}
\end{lemma}
\begin{proof}
  It is clear from the Fatou property that
  $\hat\varphi|_{L^\infty}=\varphi_0$, and in particular, it is
  proper. Since $\hat\varphi$ is a point-wise supremum of proper
  convex functions $X\mapsto \EB[XZ]-\varphi^*_0(Z)$ ($Z\in
  \dom\varphi^*_0$), $\hat\varphi$ is convex. If $X_n\in \DC_0$ for
  each $n$, and if $X_n\uparrow X$ a.s. for some $X\in L^0$, we see
  that $X\in\DC_0$ as well (since $\DC_0$ is upward solid) and that
  $\EB[XZ]=\sup_n\EB[X_nZ]$ for all $Z\in \dom\varphi_0^*\subset
  L^1_+$ by the monotone convergence theorem since $X_1^-Z\in L^1$,
  hence
  \begin{align*}
    \hat\varphi(X)&=\sup_{Z\in\dom\varphi^*_0}\left(\sup_n\EB[X_nZ]-\varphi^*_0(Z)\right)
    =\sup_n\sup_{Z\in\dom\varphi^*_0}\left(\EB[X_nZ]-\varphi^*_0(Z)\right)\\
    &=\sup_n\hat\varphi(X_n).
  \end{align*}
  Thus we have (\ref{eq:HatPhiMon1}).
\end{proof}

In the sequel, we always suppose the following without further notice:
\begin{assumption}
  \label{as:Standing}
  $\varphi_0$ is a finite-valued monotone convex function on
  $L^\infty$ satisfying the Lebesgue property (\ref{eq:LebLinfty1})
  and $\varphi_0(0)=0$.
\end{assumption}
The last assumption is just for notational simplicity.  Indeed, we can
replace $\varphi_0$ by $\varphi_0-\varphi_0(0)$ since $\varphi_0$ is
supposed to be finite, and $(\varphi,\Xs)$ is a Lebesgue extension of
$(\varphi_0,L^\infty)$ if and only if $(\varphi-\varphi_0(0),\Xs)$ is
a Lebesgue extension of $(\varphi_0-\varphi_0(0),L^\infty)$.

Suppose that $(\varphi,\Xs)$ is a Lebesgue extension of $\varphi_0$ in
the sense of Definition~\ref{dfn:LebExt}. Then observe that for any
$Y\in \Xs_+$, $|Y\wedge n|\leq Y$ and $Y\wedge n\uparrow Y$ a.s.,
hence the Lebesgue property of $\varphi$ on $\Xs$, the continuity from
below of $\hat\varphi$ on $L^0_+$ and
$\varphi|_{L^\infty}=\varphi_0=\hat\varphi|_{L^\infty}$ show that
$\varphi(Y)=\lim_n\varphi(Y\wedge n)=\lim_n\hat\varphi(Y\wedge
n)=\hat\varphi(Y)$. In particular,
\begin{lemma}
  \label{lem:Observation2}
  Let $(\varphi,\Xs)$ be a Lebesgue extension of $\varphi_0$. Then for
  any $X\in\Xs$,
  \begin{equation}
    \label{eq:ElemObs2}
    \lim_N\hat\varphi(\alpha|X|\mathds1_{\{|X|>N\}}) 
    =\lim_N\varphi(\alpha|X|\mathds1_{\{|X|>N\}})=0,\,\forall \alpha>0.  
  \end{equation}

\end{lemma}
\begin{proof}
  If $X\in \Xs$, then
  $X^\alpha_N:=\alpha|X|\mathds1_{\{|X|>N\}}\in\Xs$, $0\leq
  X^\alpha_N\leq \alpha|X|\in \Xs$ (by the solidness), and
  $X^\alpha_N\rightarrow 0$ a.s. as $N\rightarrow \infty$. Hence
  $\hat\varphi(X^\alpha_N)=\varphi(X^\alpha_N)\rightarrow 0$ by the
  Lebesgue property of $\varphi$ on $\Xs$ and
  $\hat\varphi(Y)=\varphi(Y)$ for $Y\in\Xs_+$.
\end{proof}

This leads us to the following definition:
\begin{equation}
  \label{eq:Mphiu1}
  M^{\hat\varphi}_u 
  :=\left\{X\in L^0:\, \lim_N 
    \hat\varphi\left(\alpha|X|\mathds1_{\{|X|>N\}}\right)=0,\,\forall \alpha>0
  \right\}.
\end{equation}
At the first glance, we note that this is well-defined since
$L^0_+\subset \DC_0$ and that $M^{\hat\varphi}_u$ is a solid vector
space. Indeed, the linearity follows from the observation that
$|X+Y|\mathds1_{\{|X+Y|>N\}}\leq
2|X|\mathds1_{\{|X|>N/2\}}+2|Y|\mathds1_{\{|Y|>N/2\}}$, while the
solidness is a consequence of the monotonicity of $\hat\varphi$ (and
of $x\mapsto |x|\mathds1_{\{|x|>N\}}$).

Next, we see that $\hat\varphi$ is well-defined on
$M^{\hat\varphi}_u$. Observe first from the definition
(\ref{eq:phihat1}) that
\begin{equation}
  \label{eq:YoungHat1}
  \EB[\alpha|X|Z] \leq \hat\varphi(\alpha|X|)+\varphi^*_0(Z),
  \quad \forall\alpha>0,\, X\in L^0,\,Z\in\dom\varphi^*_0.
\end{equation}
Thus $\DC_0\cap(-\DC_0)$ contains the \emph{Orlicz space} and
\emph{Orlicz heart} of $\hat\varphi$:
\begin{align}\label{eq:OrliczSp1}
  L^{\hat\varphi}:&=\left\{X\in L^0:\,\exists \alpha>0,\,
    \hat\varphi(\alpha|X|)<\infty\right\},\\
  \label{eq:OrliczHeart1}
  M^{\hat\varphi}:&=\left\{X\in L^0:\,\forall \alpha>0,\,
    \hat\varphi(\alpha|X|)<\infty\right\}.
\end{align}
Thus $\hat\varphi$ is well-defined on $L^{\hat\varphi}$ as a proper
monotone convex function, and it is finite on $M^{\hat\varphi}$ (since
$\hat\varphi(X)\leq \hat\varphi(|X|)<\infty$ if $X\in
M^{\hat\varphi}$). Also, for any $\alpha>0$, $X\in L^0$ and $N\in\NB$,
\begin{equation}
  \label{eq:HatPhiFinite1}
  \hat\varphi(\alpha|X|)\leq
  \frac12\hat\varphi(2\alpha|X|\mathds1_{\{|X|>N\}})+\frac12\varphi_0(2\alpha
  N)
\end{equation}
The second term in the right hand side is always finite since
$\varphi_0$ is supposed to be finite, and if $X\in M^{\hat\varphi}_u$,
then for any $\alpha>0$, the first term is \emph{eventually finite},
thus $M^{\hat\varphi}_u\subset M^{\hat\varphi}\subset
L^{\hat\varphi}\subset \DC_0$. Therefore, $\hat\varphi$ is
well-defined on $M^{\hat\varphi}$ as a \emph{finite-valued} monotone
convex function.
\begin{remark}
  The same argument together with (\ref{eq:ElemObs2}) tells us also
  that only \emph{finite-valued} functions can be Lebesgue extensions
  of $\varphi_0$ as long as the original function $\varphi_0$ is
  finite.
\end{remark}

\subsection{Maximum Lebesgue Extension}
\label{sec:MaxLeb}

With these preparation, we now give a positive answer to
Question~\ref{ques:1}:
\begin{theorem}
  \label{thm:MainLebExt1} 
  Suppose Assumption~\ref{as:Standing}. Then the pair
  $(\hat\varphi,M^{\hat\varphi}_u)$, defined by (\ref{eq:phihat1}) and
  (\ref{eq:Mphiu1}), is the maximum Lebesgue extension of $\varphi_0$,
  I.e.,
  \begin{enumerate}
  \item $M^{\hat\varphi}_u$ is a solid subspace of $L^0$ containing
    the constants, $\hat\varphi:M^{\hat\varphi}_u\rightarrow\RB$ is a
    monotone convex function with the Lebesgue property
    (\ref{eq:LebesgueIntro1}) on $M^{\hat\varphi}_u$ and
    $\hat\varphi|_{L^\infty}=\varphi_0$;
  \item if $(\varphi,\Xs)$ is a pair satisfying the conditions of (1),
    then $\Xs\subset M^{\hat\varphi}_u$ and
    $\varphi=\hat\varphi|_{\Xs}$.
  \end{enumerate}

\end{theorem}

A proof will be given in Section~\ref{sec:ChMeas}. Here we briefly
describe the basic idea.  We already know that $M^{\hat\varphi}_u$ is
a solid subspace of $L^0$, $\hat\varphi$ is well-defined and finite on
$M^{\hat\varphi}_u$ with $\hat\varphi|_{L^\infty}=\varphi_0$ and that
if $(\varphi,\Xs)$ is another Lebesgue extension of $\varphi_0$, then
$\Xs\subset M^{\hat\varphi}_u$ (Lemma~\ref{lem:Observation2}). It
remains only to show that $\hat\varphi$ has the Lebesgue property on
$M^{\hat\varphi}_u$ which implies also that for any $X\in \Xs\subset
M^{\hat\varphi}_u$, $\varphi(X)=\lim_n\varphi(X\mathds1_{\{|X|\leq
  n\}})=\lim_n\hat\varphi(X\mathds1_{\{|X|\leq
  n\}})=\hat\varphi(X)$. The key to the Lebesgue property of
$\hat\varphi$ on $M^{\hat\varphi}_u$ is that, after a suitable change
of measure, $M^{\hat\varphi}_u$ can be made into an
\emph{order-continuous Banach lattice} with the gauge norm induced by
$\hat\varphi$. Having established this, we can appeal to the extended
Namioka-Klee theorem that asserts that any \emph{finite} monotone
convex function on a Banach lattice is norm-continuous, and the
order-continuity of the norm then concludes the proof.

Our next interest is to understand the relation between three spaces
$M^{\hat\varphi}_u$, $M^{\hat\varphi}$ and $L^{\hat\varphi}$ as the
latter two seem more familiar. We already know, by definition,
$M^{\hat\varphi}_u\subset M^{\hat\varphi}\subset L^{\hat\varphi}$.  In
general, however, these inclusions may be strict as the following
examples illustrate.

\begin{example}[Classical Orlicz Spaces]
  \label{ex:ClassicalOrlicz}
  Let $\Phi$ and $\rho_\Phi$ be as in Example~\ref{ex:Modular} and put
  $\varphi_0=\rho_\Phi$.  Since $\rho_\Phi$ is continuous from below
  on $L^0$, we still have $\hat\varphi=\rho_\Phi$ on $L^0_+$ by
  Lemma~\ref{lem:HatPhiMon1}. Then clearly
  $M^{\hat\varphi}=M^\Phi\subset L^\Phi=L^{\hat\varphi}$, and the
  inclusion is strict if $(\Omega,\FC,\PB)$ is atomless and
  $\Phi(x)=e^{|x|}-1$. Furthermore in this case, we have
  $M^{\hat\varphi}_u=M^\Phi(=M^{\hat\varphi})$. Indeed, if $X\in
  M^\Phi$ ($\Leftrightarrow$ $\Phi(\alpha|X|)\in L^1$, $\forall
  \alpha>0$), then $\hat\varphi(\alpha|X|\mathds1_{\{|X|>N\}})
  =\EB[\Phi(\alpha|X|)\mathds1_{\{|X|>N\}}]\rightarrow 0$ by dominated
  convergence.
\end{example}

The next example shows that the inclusion $M^{\hat\rho}_u\subset
M^{\hat\rho}$ may be strict.
\begin{example}
  \label{ex:BadEx}
  Let $(\Omega,\FC)=(\NB,2^{\NB})$, with $\PB$ given by
  $\PB(\{n\})=2^{-n}$, and $(Q_k)_k$ a sequence of probabilities on
  $2^\NB$ given by $ Q_1(\{1\})=1$, $Q_n(\{1\})=1-1/n$ and
  $Q_n(\{n\})=1/n$ for each $n$. Then define
  $\varphi(X)=\sup_n\EB_{Q_n}[X]$. This is clearly monotone, convex,
  and positively homogeneous ($\varphi(\alpha X)=\alpha\varphi(X)$ for
  $\alpha\geq0$), hence $\varphi^*$ is $\{0,1\}$-valued. By
  Hahn-Banach, we see that $\varphi^*(Z)=0$ if and only if
  $Z\in\overline{\mathrm{conv}}(dQ_n/d\PB,n\in\NB)=:\ZC$, and it is
  clear that $\ZC$ is uniformly integrable (thus weakly compact), and
  $\varphi$ has the Lebesgue property on $L^\infty\simeq
  l^\infty$. Also, $\hat\varphi(X)=\sup_n\EB_{Q_n}[X]$ is valid for
  all $X\geq 0$.

  Now consider a non-negative function $X(k)=k$. Then
  $\EB_{Q_n}[X]=(1-1/n)+n\cdot(1/n)=2-1/n$, hence $\hat\varphi(\alpha
  |X|)=\alpha\sup_n\EB_{Q_n}[X]=2\alpha<\infty$, thus $X\in
  M^{\hat\varphi}$. On the other hand,
  $\EB_{Q_n}[X\mathds1_{\{X>N\}}]=\mathds1_{\{n>N\}}$, thus for any
  $\alpha>0$, $\hat\varphi(\alpha |X|\mathds1_{\{|X|>N\}})=\alpha
  \sup_n\EB_{Q_n}[X\mathds1_{\{X>N\}}]\equiv \alpha$ for all $
  N$. Hence $X\not\in M^{\hat\varphi}_u$, and consequently,
  $M^{\hat\varphi}_u\subsetneq M^{\hat\varphi}$.
\end{example}

We now state our second result, which well-explains the reason for the
subscript ``$u$''.

\begin{theorem}
  \label{thm:MuCharact}
  For $X\in M^{\hat\varphi}$, the following three conditions are
  equivalent:
  \begin{enumerate}
  \item $X\in M^{\hat\varphi}_u$;
  \item $\{XZ:\, \varphi^*_0(Z)\leq c\}$ is uniformly integrable for
    all $c>0$;

  \item for some $\varepsilon>0$,
    $\sup_{Z\in\dom\varphi^*_0}(\EB[(|X|\vee
    \varepsilon)YZ]-\varphi^*_0(Z))$ is attained for all $Y\in
    L^\infty$.
  \end{enumerate}
  Moreover, these three equivalent conditions imply that
  \begin{equation}
    \label{eq:MaxAttained1}
    \hat\varphi(X)=\max_{Z\in \dom\varphi^*_0}(\EB[XZ]-\varphi^*_0(Z)),
  \end{equation}
  i.e., the supremum in (\ref{eq:phihat1}) is attained.
\end{theorem}
We prove this theorem in Section~\ref{sec:MvarphiU}.

\subsection{Characterization of Lebesgue Property on Solid Spaces}
\label{sec:JST}

Here we apply our results to obtain a characterization of the Lebesgue
property of finite monotone convex functions on \emph{arbitrary solid
  spaces} in the spirit of Theorem~\ref{thm:JSTLinfty} for the
$L^\infty$ case.  Suppose we are given a solid space $\Xs\subset L^0$
and a finite monotone convex function $\psi:\Xs\rightarrow \RB$ with
the \emph{Fatou property} (not Lebesgue at now). Then the restriction
$\psi_\infty:=\psi|_{L^\infty}$ is a finite monotone convex function
on $L^\infty$ having the Fatou property too, and putting
$\psi_\infty^*(Z)=\sup_{X\in L^\infty}(\EB[XZ]-\psi_\infty(X))$,
\begin{equation}
  \label{eq:psiHat}
  \hat\psi(X):=\sup_{Z\in \dom\psi_\infty^*}(\EB[XZ]-\psi_\infty^*(Z)),
\end{equation}
defines an extension of $\psi_\infty$ to $\DC_\psi:=\{X\in L^0:\,
X^-Z\in L^1,\,\forall Z\in\dom\psi_\infty^*\}\supset L^0_+\cup
L^\infty$ by the Fatou property. Note that the monotonicity
($\Rightarrow$ $\dom\psi^*_\infty\subset L^1_+$) and the finiteness of
$\psi$ on the whole $\Xs$ implies $\Xs\subset
\DC_\psi\cap(-\DC_\psi)$, or equivalently,
\begin{equation}
  \label{eq:OrderContiDual1}
  XZ\in L^1,\,\forall X\in \Xs,\,Z\in \dom\psi^*_\infty.
\end{equation}
Thus $\hat\psi$ is well-defined on $\Xs$ in particular. Indeed,
observe that $\EB[|X|Z]-\psi_\infty^*(Z)=\sup_n(\EB[|X|\wedge n
Z]-\psi_\infty^*(Z))\leq \sup_n\psi(|X|\wedge n)\leq\psi(|X|)<\infty$
for $X\in\Xs$ and $Z\in \dom\psi_\infty^*$ where we used Young's
inequality for the pair $(\psi|_{L^\infty},\psi^*_\infty)$.

On the other hand, the original $(\psi,\Xs)$ is also an extension of
$\psi_\infty$ since the latter is the restriction of $\psi$. Then
close comparisons of these two extensions using
Theorems~\ref{thm:MainLebExt1}~and~\ref{thm:MuCharact} yield the
following generalization of the JST Theorem~\ref{thm:JSTLinfty}:
\begin{theorem}[Generalization of JST-Theorem \citep{jouini06:_law_fatou}]
  \label{thm:JSTGeneral}
  Let $\Xs\subset L^0$ be a solid space containing the constants and
  $\psi:\Xs\rightarrow\RB$ be a finite-valued monotone convex function
  satisfying the Fatou property (\ref{eq:FatouX2}) on $\Xs$. Then the
  following are equivalent:
  \begin{enumerate}
  \item $\psi$ has the Lebesgue property (\ref{eq:LebX2}) on $\Xs$;
  \item for all $X\in \Xs$ and $c>0$, $\{XZ: \psi_\infty^*(Z)\leq c\}$
    is uniformly integrable;
  \item the supremum
    $\sup_{Z\in\dom\psi_\infty^*}(\EB[XZ]-\psi_\infty^*(Z))$ is finite
    and attained for all $X\in \Xs$;
  \item it holds that $\psi(X)=\max_{Z\in
      \dom\psi_\infty^*}(\EB[XZ]-\psi^*_\infty(Z))$, $\forall
    X\in\Xs$.
  \end{enumerate}
\end{theorem}
A proof is given in Section~\ref{sec:JST}.  Note that (4) is not a
paraphrasing of (3) since it is not \emph{a priori} assumed that
$\psi(X)=\sup_{Z\in
  \dom\psi^*_\infty}(\EB[XZ]-\psi^*_\infty(Z))=\hat\psi(X)$ for all
$X\in \Xs$.

When $\Xs=L^\infty$, then $\psi=\hat\psi$ hence (3) $\Leftrightarrow$
(4) is trivial, and (2) is equivalent to saying that $\{Z\in L^1:\,
\psi^*_\infty(Z)\leq c\}$ is $\sigma(L^1,L^\infty)$-compact for all
$c>0$ by the Dunford-Pettis theorem. Thus, in this case,
Theorem~\ref{thm:JSTGeneral} is nothing but
Theorem~\ref{thm:JSTLinfty} which is essentially due to
\citep{jouini06:_law_fatou} and \citep{MR2648597}.  Some other
(partial) generalizations of Theorem~\ref{thm:JSTLinfty} have been
obtained in literature, so we briefly discuss here some key features
of \emph{our} version.

\paragraph{\bf Generality of the space $\Xs$}

The only a priori assumption on the space $\Xs$ is that it is a solid
vector subspace (ideal) of $L^0$ containing the constants. All Orlicz
spaces and hearts as well as $L^p$ with $p\in [0,\infty]$ are of this
type. Note also that without the solidness, the Lebesgue and Fatou
properties do not ``well'' make sense.

\paragraph{\bf Our formulation is ``universal''}

We note that topological qualifications (of $\Xs$ and $\psi$) are
absent in our formulation: $\psi^*_\infty=(\psi|_{L^\infty})^*$ is
used instead of the conjugate of $\psi$ on the topological dual of
$\Xs$, the \emph{inf-compactness} of the conjugate is alternatively
stated in a form of uniform integrability, and the Fatou and Lebesgue
properties are regularities in terms of order structure. These
ingredients are in some sense more ``universal'' than the topological
counter-parts. It should also be emphasized that our characterization
is still quite explicit even though it does not rely on the
topological nature of $\Xs$.
% and if $\psi$ is Lebesgue on $\Xs$, then under a mild separation
% assumption (see below) $\Xs$ can \emph{a fortiori} be made into a
% nice topological space as a subspace of $M^{\hat\psi}_u$, but the a
% priori given topology (if any) does not matter.

\begin{remark}\label{rem:Alternative}
  Theorem~\ref{thm:JSTGeneral} can be alternatively stated in terms of
  the \emph{order-continuous dual} of $\Xs$, which is regarded, under
  our assumption on $\Xs$, as the set
  \begin{equation}
    \label{eq:OderContiDual}
    \Xs^\sim_n=\{Z\in L^0:\, XZ\in L^1,\,\forall X\in \Xs\}.
  \end{equation}
  via the identification of $Z$ and the order-continuous linear
  functional $X\mapsto \EB[XZ]$.  Observe that
  $\dom\psi^*_\infty\subset \Xs^\sim_n\subset L^1$ by $L^\infty\subset
  \Xs$ and (\ref{eq:OrderContiDual1}), thus ``$\dom\psi^*_\infty$'' in
  the statements can be replaced by $\Xs^\sim_n$. In particular, the
  Lebesgue property of $\psi$ implies the ``simplified dual
  representation'' on $\Xs^\sim_n$ with the penalty function
  $\psi^*_\infty$ (see \citep{MR2648595}) without
  any structural assumption on the space $\Xs$ (than being an ideal of
  $L^0$). Also, item (2) is in fact equivalent to the relative
  compactness of all the level sets $\{Z\in\Xs^\sim_n:\,
  \psi^*_\infty(Z)\leq c\}$ for the weak topology
  $\sigma(\Xs^\sim_n,\Xs)$, which is a (well-defined) locally convex
  Hausdorff topology as long as $\Xs$ contains the constants as we are
  assuming.
\end{remark}

Given the above discussion, it seems also natural (and more common) to
characterize the Lebesgue property in the form of
Theorem~\ref{thm:JSTGeneral} but with the conjugate
\begin{equation}
  \label{eq:ConjPsiUsual}
  \psi^*(Z):=\sup_{X\in \Xs}(\EB[XZ]-\psi(X)),\,Z\in \Xs^\sim_n
\end{equation}
instead of $\psi^*_\infty$. In fact, the equivalence of (1) -- (4) in
Theorem~\ref{thm:JSTGeneral} remains true (see
\citep{owari13:_lebes_monot_convex_funct}) with $\psi^*$ instead of
$\psi^*_\infty$ if (a) $\Xs\subset L^1(\QB)$ for some $\QB\sim\PB$ and
if (b) \emph{$\psi$ is a priori assumed to be
  $\sigma(\Xs,\Xs^\sim_n)$-lower semicontinuous or equivalently}
\begin{equation}
  \label{eq:FatouFrechet}
  \psi(X)
  =\sup_{Z\in\Xs^\sim_n}(\EB[XZ]-\psi^*(Z)),\,\forall X\in\Xs.
\end{equation}
Here (a) is rather technical, which says simply that $\Xs^\sim_n$
separates $\Xs$, and only the equivalence ``$\QB\sim\PB$'' is
essential since that $\Xs$ accommodates a \emph{finite} monotone
convex function $\psi$ with the Fatou property already implies the
existence of $\QB\ll \PB$ such that $\Xs\subset L^1(\QB)$. The
assumption (b) ($\Leftrightarrow$ (\ref{eq:FatouFrechet})) implies the
Fatou property (see \citep[][Proposition~1]{MR2648595}), but the
converse is not generally true, and (b) may not be easy to check. In
some ``good'' cases, however, (b) is actually equivalent to the Fatou
property, and the ``good'' cases include $\Xs=L^\infty$ ($\Rightarrow$
$\Xs^\sim_n=L^1$), $\Xs=M^\Phi$ with finite Young function $\Phi$
(then $\Xs^\sim_n=L^{\Phi^*}$), and $\Xs=L^\Phi$ with $\Phi$
satisfying the so-called $\Delta_2$-condition (then
$L^\Phi=M^\Phi$). For more general $\Xs$, however, it is still open
when the Fatou property implies the $\sigma(\Xs,\Xs^\sim_n)$-lower
semicontinuity for all convex functions.

\begin{remark}\label{rem:BF}
  The above question is equivalent to asking if all \emph{order closed
    convex} subsets of $\Xs$ are $\sigma(\Xs,\Xs^\sim_n)$-closed. This
  is true as soon as it is shown that any
  $\sigma(\Xs,\Xs^\sim_n)$-convergent net $(X_\alpha)_\alpha$ in $\Xs$
  admits a sequence of indices $(\alpha_n)_n$ as well as a sequence
  $\tilde X_n\in \mathrm{conv}(X_{\alpha_n},X_{\alpha_{n+1}},\ldots)$
  which converges \emph{in order} to the same limit. In
  \citep[][Lemma~6 and Corollary~4]{MR2648595}, it is claimed that the
  last property is true whenever (adapted to our notation) $\Xs$ is
  (lattice homomorphic to) and ideal of $L^1$ (hence of
  $L^0$). Unfortunately, however, their proof has an error. There it
  is shown that with the above assumption, any
  $\sigma(\Xs,\Xs^\sim_n)$-convergent net $(X_\alpha)_\alpha$ admits a
  sequence $(\tilde X_n)_n$ of forward convex combinations of the
  above form which, as a sequence in $L^1$, converges \emph{in order
    of $L^1$} to the same limit. This part is correct. Then it was
  concluded that $\tilde X_n$, \emph{as a sequence in $\Xs$},
  converges \emph{in order of $\Xs$} to the same limit. The last part
  is not true at least solely from the assumptions imposed on
  $\Xs$. In general, whenever $\Xs$ is an ideal of $L^0$, the order
  convergence in $\Xs$ of a sequence $(X_n)_n$ is equivalent to the
  dominated a.s. convergence (i.e., $X_n\rightarrow X$ a.s. and
  $\exists Y\in \Xs_+$ with $|X_n|\leq Y$ ($\forall n$)). The
  a.s. convergence is universal (which is common to all ideals of
  $L^0$), while being dominated in $\Xs$ is not universal. For a
  trivial example, picking $Z\in L^1_+\setminus L^\infty$, the
  sequence $X_n=Z\wedge n$ which lies in $L^\infty$ converges in order
  in $L^1$ to $Z$, but does not converge in order in $L^\infty$. What
  we need to fill the gap is still an open question (for us).
\end{remark}
% Regarding the last point, \citep[][Lemma~6 and
% Corollary~4]{MR2648595} claim that the Fatou
%  implies (\ref{eq:FatouFrechet}) whenever $\Xs \subset L^0$ is
% a locally convex Fréchet lattice injected into $L^1$ by a
% topologically continuous lattice homomorphism, and this is especially
% the case if $\Xs$ is an Orlicz space $L^\Phi$.  Unfortunately,
% however, the proof given there is not correct, and it is still open
% when the Fatou property implies the $\sigma(\Xs,\Xs^\sim_n)$-lower
% semicontinuity.

% The equivalence of the Fatou property and (\ref{eq:FatouFrechet}) is
% discussed in \citep{MR2648595} where it is claimed
% that it is true for all \emph{locally convex-Fréchet lattices} of
% random variables injected into $L^1$ by a topologically continuous
% lattice homomorphism (see \citep[][Lemma~6 and
% Corollary~4]{MR2648595}), and this is the case if
% $\Xs=L^\Phi$ (resp.  $M^\Phi$) with an arbitrary (resp. finite) Young
% function $\Phi$. But unfortunately, the proof given there is not
% correct, and it is still open when the Fatou property implies the
% $\sigma(\Xs,\Xs^\sim_n)$-lower semicontinuity.

\begin{remark}
  When $\Phi^*$ is finite, \citep{orihuela_ruiz12:_lebes_orlic}
  recently obtained the equivalence of (1) -- (4) with $\psi^*$ for
  $\Xs=L^\Phi$, but with an even stronger assumption than
  (\ref{eq:FatouFrechet}) that $\psi$ is
  $\sigma(L^\Phi,M^{\Phi^*})$-lower semicontinuous (note in this case
  that $\Xs^\sim_n=L^{\Phi^*}$ which is strictly bigger than
  $M^{\Phi^*}$ if the probability space is atomless and $\Phi$ does
  not satisfy the $\Delta_2$-condition). When $\Xs$ is a locally
  convex Fréchet lattice, the implication (1) $\Rightarrow$ (4) is
  (implicitly) contained in \citep[][Lemma~7]{MR2648595}. For the
  equivalence of (1) -- (4) with $\psi^*$ for general solid space
  $\Xs$ containing the constants under the assumptions (a) and (b)
  above, see \citep{owari13:_lebes_monot_convex_funct}.
\end{remark}

Note that with the standing assumptions of
Theorem~\ref{thm:JSTGeneral} only, the inequality $\EB[XZ]\leq
\psi(X)+\psi^*_\infty(Z)$ is not guaranteed \emph{for all $X\in\Xs$}
and $Z\in \Xs^\sim_n$ (it is true for $X\in \Xs_+\cup
L^\infty$). However, if $\psi$ has the Lebesgue property, we see that
$\EB[XZ] =\lim_n\EB[X\mathds1_{\{|X|\leq n\}}Z]\leq \limsup_n
\psi(X\mathds1_{\{|X|\leq n\}})+\psi^*_\infty(Z)
=\psi(X)+\psi^*_\infty(Z)$. Thus (1) $\Rightarrow$ (4) shows that

\begin{corollary}
  For a finite monotone convex function $\psi$ on a solid vector space
  $\Xs\subset L^0$, the Lebesgue property implies the existence of a
  $\sigma$-additive subgradient of $\psi$ at everywhere on $\Xs$, that
  is, for all $X\in \Xs$, there exists a $Z\in\Xs^\sim_n\subset L^1$
  such that
  \begin{align*}
    \EB[XZ]-\psi(X)\geq \EB[YZ]-\psi(Y),\,\forall Y\in \Xs.
  \end{align*}
\end{corollary}

\section{Analysis of the space $M^{\hat\varphi}_u$ and Proof of
  Theorem~\ref{thm:MainLebExt1}}
\label{sec:ProofMaxLebesgue}

\emph{Throughout this section, Assumption~\ref{as:Standing} is in
  force unless the contrary is explicitly stated}. The key to the
proof of Theorem~\ref{thm:MainLebExt1} is the analysis of the
Orlicz-type space $M^{\hat\varphi}_u$.

\subsection{The Gauge of $\hat\varphi$}
\label{sec:Gauge}

Let us define the \emph{gauge} of the monotone convex function
$\hat\varphi$:
\begin{equation}
  \label{eq:Gauge}
  \|X\|_{\hat\varphi}:=\inf\{\lambda>0:\, \hat\varphi(|X|/\lambda)\leq 1\},\,\forall X\in L^0,
\end{equation}
with the convention $\inf\emptyset=+\infty$.  In analogy to the
Luxemburg norms of usual Orlicz spaces, we see easily that for any
$X,Y\in L^0$ and $\alpha \in\RB$,
\begin{equation}
  \label{eq:Seminorm1}
  \|\alpha X\|_{\hat\varphi}=|\alpha|\|X\|_{\hat\varphi},\, \|X+Y\|_{\hat\varphi}\leq \|X\|_{\hat\varphi}+\|Y\|_{\hat\varphi} 
  \text{ and } \|X\|_{\hat\varphi}\leq \|Y\|_{\hat\varphi}\text{ if }|X|\leq |Y|.
\end{equation}
Indeed, the first (resp. last) one follows from a change of variable
$\lambda'=\lambda/\alpha$ (resp. monotonicity of $\hat\varphi$), while
the convexity and monotonicity of $\hat\varphi$ implies that for any
$\alpha\in (0,1)$,
\begin{align*}
  \hat\varphi\left(\frac{|\alpha X+(1-\alpha)
      Y|}{\alpha\lambda+(1-\alpha)\lambda'}\right)\leq \frac{\alpha
    \lambda}{\alpha\lambda+(1-\alpha)\lambda'}\hat\varphi\left(\frac{|X|}\lambda\right)
  +\frac{(1-\alpha)\lambda'}{\alpha\lambda+(1-\alpha)\lambda'}
  \hat\varphi\left(\frac{|Y|}{\lambda'}\right),
\end{align*}
hence $\{\alpha \lambda+(1-\alpha)\lambda':\,
\lambda,\lambda'>0,\,\hat\varphi(|X|/\lambda),\,\hat\varphi(|Y|/\lambda')\leq
1\} \subset \{\beta>0:\, \hat\varphi(|\alpha X+(1-\alpha)Y|/\beta)\leq
1\}$. We have also that
\begin{align}\label{eq:SeminormFinite}
  & \|X\|_{\hat\varphi}<\infty \text{ if and only if } X\in
  L^{\hat\varphi};\\
  \label{eq:SeminormZero}
  &\|X\|_{\hat\varphi}=0\text{ if and only if }
  \hat\varphi(\alpha|X|)=0,\,\forall  \alpha>0;\\
  \label{eq:SeminormConvergence}
  &\|X_n\|_{\hat\varphi}\rightarrow 0\text{ if and only if
  }\hat\varphi(\alpha|X_n|)\rightarrow 0,\,\forall \alpha>0.
\end{align}
The necessity of (\ref{eq:SeminormFinite}) is clear from the
definition while the convexity of $\hat\varphi$ and $\hat\varphi(0)=0$
imply that $\hat\varphi(\varepsilon\alpha|X|)\leq
\varepsilon\hat\varphi(\alpha|X|)
=\varepsilon\hat\varphi(\alpha|X|)\downarrow0$ whenever
$\hat\varphi(\alpha|X|)<\infty$.  The sufficiency of
(\ref{eq:SeminormZero}) is again immediate from (\ref{eq:Gauge}), and
$\|X\|_{\hat\varphi}=0$ implies that $\hat\varphi(\alpha|X|)\leq
\varepsilon\hat\varphi((\alpha/\varepsilon)|X|)\leq \varepsilon$ for
any $\varepsilon\in (0,1)$, hence $\hat\varphi(\alpha|X|)=0$. Finally,
(\ref{eq:SeminormConvergence}) follows from the relations
$\|X\|_{\hat\varphi}<\varepsilon$ $\Rightarrow$
$\hat\varphi(|X|/\varepsilon)\leq 1$ $\Rightarrow$
$\|X\|_{\hat\varphi}\leq \varepsilon$, and $\hat\varphi(\alpha|X|)\leq
\varepsilon\alpha \hat\varphi(|X|/\varepsilon)\leq \varepsilon\alpha$
if $\varepsilon<1/\alpha$.

In general, any $\RB$-valued function $p$ on a Riesz space verifying
the three conditions of (\ref{eq:Seminorm1}) is called a
\emph{lattice} seminorm. In view of (\ref{eq:SeminormFinite}), we have
seen that $\|\cdot\|_{\hat\varphi}$ is a lattice seminorm on
$L^{\hat\varphi}$ (hence on $M^{\hat\varphi}$ and $M^{\hat\varphi}_u$
as well).

Note that we have used only three properties of $\hat\varphi$ so far,
namely, convexity, monotonicity and $\hat\varphi(0)=0$, so the
arguments above still work for any monotone convex function on $L^0_+$
null at the origin. Now the continuity from below of $\hat\varphi$
(Lemma~\ref{lem:HatPhiMon1}) shows:
\begin{lemma}
  \label{lem:LatticeSemiNorm1}
  For any $\alpha>0$, $\|X\|_{\hat\varphi}\leq\alpha$ if and only if
  $\hat\varphi(|X|/\alpha)\leq 1$, and
  \begin{equation}
    \label{eq:NormLSC}
    X_n\rightarrow X \text{ a.s. }\Rightarrow\, \|X\|_{\hat\varphi}\leq
    \liminf_n\|X_n\|_{\hat\varphi}.
  \end{equation}

\end{lemma}
\begin{proof}
  The sufficiency of the first claim is clear from (\ref{eq:Gauge}),
  while the monotonicity and continuity from below of $\hat\varphi$
  imply that for any $\alpha>0$,
  \begin{align*}
    \alpha>\|X\|_{\hat\varphi}\,\Rightarrow\,
    \hat\varphi\left(|X|/\alpha\right)
    =\lim_n\hat\varphi\left(\frac{|X|}{\alpha+1/n}\right) \leq
    \lim_n\hat\varphi\left(\frac{|X|}{\|X\|_{\hat\varphi}+1/n}\right)
    \leq 1.
  \end{align*}

  For (\ref{eq:NormLSC}), we may suppose $\|X\|_{\hat\varphi}>0$
  (otherwise trivial). Put $Y_n:=\inf_{k\geq n}|X_k|$ and note that
  $0\leq Y_n\uparrow |X|$ by $X_n\rightarrow X$. Then for any
  $\varepsilon\in (0,\|X\|_{\hat\varphi})$,
  \begin{align*}
    1<\hat\varphi\left(\frac{|X|}{\|X\|_{\hat\varphi}-\varepsilon}\right)
    =\lim_n\hat\varphi\left(\frac{Y_n}{\|X\|_{\hat\varphi}-\varepsilon}\right),
  \end{align*}
  which implies in view of the first claim that
  $\|Y_n\|_{\hat\varphi}>\|X\|_{\hat\varphi}-\varepsilon$ for large
  enough $n$, thus we deduce that $\liminf_n\|X_n\|_{\hat\varphi}\geq
  \sup_n\|\inf_{k\geq
    n}|X_k|\|_{\hat\varphi}=\sup_n\|Y_n\|_{\hat\varphi}\geq
  \|X\|_{\hat\varphi}-\varepsilon$. Since $\varepsilon>0$ is
  arbitrary, we have (\ref{eq:NormLSC}).
\end{proof}

The next one is crucial.
\begin{lemma}
  \label{lem:OrderContiSeminorm}
  The lattice seminorm $\|\cdot\|_{\hat\varphi}$ is order-continuous
  on $M^{\hat\varphi}_u$, i.e.,
  \begin{equation}
    \label{eq:OrderContiSeminorm1}
    \|X_n\|_{\hat\varphi}\rightarrow 0 \text{ whenever }    \exists Y\in M^{\hat\varphi}_u\text{ with }
    |X_n|\leq |Y|\,(\forall n)\text{ and } X_n\rightarrow 0\text{, a.s. } 
    % \Rightarrow\, \|X_n\|_{\hat\varphi}\rightarrow 0.
    % \exists Y\in M^{\hat\varphi}_u\text{ with }
    % |X_n|\leq |Y|\,(\forall n)\text{ and } X_n\rightarrow 0\text{,
    % a.s. }
    % \Rightarrow\, \|X_n\|_{\hat\varphi}\rightarrow 0.
  \end{equation}

\end{lemma}
\begin{proof}
  Let $(X_n)_n\subset M^{\hat\varphi}_u$ be dominated by $Y\in
  M^{\hat\varphi}_u$ and converges a.s. to $0$. Then
  \begin{align*}
    \|X_n\|_{\hat\varphi}&\leq
    \|X_n\mathds1_{\{|Y|>N\}}\|_{\hat\varphi}+\|X_n\mathds1_{\{|Y|\leq
      N\}}\|_{\hat\varphi}\leq
    \|Y\mathds1_{\{|Y|>N\}}\|_{\hat\varphi}+\|X_n\mathds1_{\{|Y|\leq
      N\}}\|_{\hat\varphi}.
  \end{align*}
  We claim that (1)
  $\|Y\mathds1_{\{|Y|>N\}}\|_{\hat\varphi}\stackrel{N}\rightarrow 0$,
  and (2) for each fixed $N$, $\|X_n\mathds1_{\{|Y|\leq
    N\}}\|_{\hat\varphi}\stackrel{n}\rightarrow0$, then
  (\ref{eq:OrderContiSeminorm1}) follows by a diagonal argument.  In
  fact, (1) is equivalent in view of (\ref{eq:SeminormConvergence}) to
  saying that $\hat\varphi(\alpha|Y|\mathds1_{\{|Y|>N\}})\rightarrow
  0$ for all $\alpha>0$, which is nothing but the definition of $Y$
  being an element of $M^{\hat\varphi}_u$. As for (2), we note that
  the sequence $Z_n^N:=X_n\mathds1_{\{|Y|\leq N\}}$ satisfies
  $\sup_n\|Z_n^N\|_\infty\leq N<\infty$ (since $|X_n|\leq |Y|$ by
  assumption) and $Z_n^N\rightarrow 0$ a.s. ($n\rightarrow
  \infty$). Thus the Lebesgue property of
  $\varphi_0=\hat\varphi|_{L^\infty}$ shows that
  $\hat\varphi(\alpha|Z^N_n|)=\varphi_0(\alpha|Z^N_n|)\rightarrow 0$
  for all $\alpha>0$, hence $\|Z_n^N\|_{\hat\varphi}\rightarrow 0$ by
  (\ref{eq:SeminormConvergence}).
\end{proof}

We now characterize the space $M^{\hat\varphi}_u$ in terms of the
gauge seminorm $\|\cdot\|_{\hat\varphi}$.
\begin{lemma}
  \label{lem:ChMphiUNorm1}
  For any $X\in L^0$, the following are equivalent:
  \begin{enumerate}
  \item $X\in M^{\hat\varphi}_u$;
  \item $\lim_N\|X\mathds1_{\{|X|>N\}}\|_{\hat\varphi}=0$;
  \item $\lim_n\|X\mathds1_{A_n}\|_{\hat\varphi}=0$ whenever
      $\PB(A_n)\downarrow 0$.
  \end{enumerate}

\end{lemma}
\begin{proof}
  (3) $\Rightarrow$ (2) is clear, and (2) $\Rightarrow$ (1) was
  already proved in the proof of Lemma~\ref{lem:OrderContiSeminorm}.
  If $X\in M^{\hat\varphi}_u$, then $Y_n:=X\mathds1_{A_n}\in
  M^{\hat\varphi}_u$, $|Y_n|\leq |X|$ and $Y_n\rightarrow 0$
  a.s. whenever $\PB(A_n)\rightarrow0$. Thus (1) $\Rightarrow$ (3)
  follows from Lemma~\ref{lem:OrderContiSeminorm}.
\end{proof}

Finally, we have the following inequality:
\begin{lemma}
  \label{lem:SeminormEmbedding}
  For any $X\in L^0$ and $Z\in \dom\varphi^*_0$,
  \begin{equation}
    \label{eq:SeminormEmbedding}
    \EB[|X|Z]\leq (1+\varphi^*_0(Z))\|X\|_{\hat\varphi}.
  \end{equation}

\end{lemma}
\begin{proof}
  We may assume $\|X\|_{\hat\varphi}<\infty$ (otherwise trivial).
  Then (\ref{eq:YoungHat1}) shows $1\geq
  \hat\varphi\left(|X|/\alpha\right)\geq
  \EB[|X|Z/\alpha]-\varphi^*_0(Z)$ for any
  $\alpha>\|X\|_{\hat\varphi}$ and $Z\in \dom\varphi^*_0$, thus
  rearranging the terms,
   \begin{align*}
     \EB[|X|Z]\leq
     (1+\varphi^*_0(Z))(\|X\|_{\hat\varphi}+\varepsilon),\,\forall
     \varepsilon>0,
  \end{align*}
  and we have (\ref{eq:SeminormEmbedding}) by letting
  $\varepsilon\downarrow 0$.
\end{proof}

\subsection{Quotient via a Change of Measure}
\label{sec:ChMeas}

We already know that $(M^{\hat\varphi}_u,\|\cdot\|_{\hat\varphi})$ is
a \emph{semi}-normed Riesz space with the \emph{order-continuous}
lattice seminorm, and $\hat\varphi$ is a \emph{finite} monotone convex
function on it. But $\|\cdot\|_{\hat\varphi}$ is not generally a
\emph{norm}, i.e., $\|X\|_{\hat\varphi}=0$ does not imply $X=0$ as an
element of $M^{\hat\varphi}_u$ (or in $L^0$), thus we cannot directly
conclude that $M^{\hat\varphi}_u$ is an order-continuous Banach
lattice. A standard way of tackling this kind of difficulty is to take
the quotient by the relation induced by $\|X\|_{\hat\varphi}=0$. We
shall do this through a suitable change of probability.

\begin{lemma}
  \label{lem:SensitivityMaximal1}
  There exists a $\Zh\in\dom\varphi^*_0$ such that for any $A\in\FC$,
  \begin{equation}
    \label{eq:HalmosSavage1}
    \EB[\Zh\mathds1_A]=0\,\Rightarrow \, \EB[Z\mathds1_A]=0,\,\forall Z\in\dom\varphi^*_0.
  \end{equation}
  Then putting $d\QB/d\PB=\cH\Zh$ (with $\cH=\EB[\Zh]^{-1}$), $\QB$ is
  a probability measure such that
  \begin{equation}
    \label{eq:ProbQ}
    \varphi^*_0(\cH d\QB/d\PB)<\infty,\text{ and }    \QB(|X|>0)=0\,
    \Leftrightarrow\, \hat\varphi(\alpha|X|)=0,\,\forall \alpha>0.
  \end{equation}
\end{lemma}
\begin{remark}
  \label{rem:SensitiveFatouEnough}
  As we shall see in the proof, this lemma does not need the Lebesgue
  property of $\varphi^0$; the Fatou property is enough.
\end{remark}
\begin{proof}
  We first construct a $\Zh\in \dom\varphi_0^*\subset L^1$ such that
  \begin{equation}
    \label{eq:SensitivityMaximal}
    \varphi^*_0(\Zh)\leq 1\text{ and }  
    \PB(\Zh>0)=\max\{\PB(Z>0):\, Z\in \dom\varphi_0^*,\,\varphi^*_0(Z)\leq 1\}.
  \end{equation}
  The set $\Lambda:=\{Z\in \dom\varphi_0^*: \varphi_0^*(Z)\leq 1\}$ is
  convex, norm closed in $L^1$ by the lower semicontinuity of
  $\varphi^*_0$, and is norm bounded since $\EB[|Z|]=\EB[Z]\leq
  \varphi_0(1)+1$ for all $Z\in \Lambda$. Thus for any sequences
  $(Z_n)_n\subset\Lambda$ and $(\alpha_n)\subset \RB_+$ with
  $\sum_n\alpha_n=1$, the series $Z:=\sum_n\alpha_nZ_n$ is absolutely
  convergent in $L^1$, and we have in fact $Z\in \Lambda$. Indeed,
 \begin{align*}
   \varphi_0^*(Z)&=\sup_{X\in
     L^\infty}\left(\EB[XZ]-\varphi_0(X)\right)=\sup_{X\in
     L^\infty}\left(\sum_n\alpha_n\EB[XZ_n]-\varphi_0(X)\right)\\
   &\leq\sum_n\alpha_n\sup_{X'\in
     L^\infty}\left(\EB[X'Z_n]-\varphi_0(X')\right)=\sum_n\varphi^*_0(Z_n)\leq
   1.
  \end{align*}
  In other words, $\Lambda$ is countably convex. Then choosing a
  sequence $(Z_n)_n\subset \Lambda$ so that $\PB(Z_n>0)\uparrow
  \sup_{Z\in \Lambda}\PB(Z>0)$, $\Zh:=\sum_n2^{-n}Z_n\in \Lambda$ and
  we have (\ref{eq:SensitivityMaximal}).

  We check that this $\Zh$ satisfies (\ref{eq:HalmosSavage1}). Indeed,
  if there exists a $Z\in \dom\varphi^*_0$ and $A\in\FC$ such that
  $\EB[\Zh\mathds1_A]=0$ and $\EB[Z\mathds1_A]>0$, we see that
  $\PB(\Zh=0,\,Z>0)>0$, $\varepsilon Z\in \Lambda$ for some small
  $\varepsilon>0$ since $\varphi^*_0(0)=0$ and $\bar
  Z:=(\Zh+\varepsilon Z)/2\in \Lambda$ satisfies
  \begin{align*}
    \PB(\bar Z>0)=\PB(\Zh>0)+\PB(Z>0,\,\Zh=0)>\PB(\Zh>0).
  \end{align*}
  This contradicts to (\ref{eq:SensitivityMaximal}).

  Finally, putting $d\QB/d\PB=\Zh/\EB[\Zh]$, the first condition of
  (\ref{eq:ProbQ}) is clear. For the second, if
  $\QB(|X|>0)=\EB[\Zh\mathds1_{\{|X|>0\}}]=0$, then $\EB[|X|Z]=0$ for
  all $Z\in\dom\varphi^*_0$, hence $\hat\varphi(\alpha\mathbb|X|)
  =\sup_{Z\in\dom\varphi^*_0}(\alpha\EB[Z|X|]-\varphi^*_0(Z))=0$ for
  all $\alpha>0$. On the other hand, if $\hat\varphi(\alpha|X|)=0$ for
  all $\alpha>0$, then $\alpha \EB[\Zh|X|]\leq
  \hat\varphi(\alpha|X|)+\varphi^*_0(\Zh)\leq 1$ for all $\alpha>0$,
  thus $\EB[|X|\Zh]=0$, and consequently $\QB(|X|>0)=0$.
\end{proof}

By (\ref{eq:ProbQ}), we see that $\|X\|_{\hat\varphi}=0$ if and only
if $X=0$, $\QB$-a.s. Let
\begin{equation}
  \label{eq:NPQ}
  \NC_{\PB}(\QB):=\{X\in L^0:\, X=0,\,\QB\text{-a.s.}\}=\{X\in L^0:\, \hat\varphi(\alpha|X|)=0,\,\forall \alpha>0\}.
\end{equation}
The quotient space $L^0/\|\cdot\|_{\hat\varphi}=L^0/\NC_\PB(\QB)$ is
(lattice isomorphic to) the space $L^0(\QB)$ of equivalence classes
\emph{modulo $\QB$-a.s. equality} of measurable functions ordered by
the $\QB$-a.s. inequality (remember that $L^0=L^0(\PB)$ also is the
space of classes but \emph{modulo $\PB$-a.s. equality}). All we need
is the following intuitively obvious lemma:
\begin{lemma}
  \label{lem:Sensitive1}
  There exists an onto linear mapping $\pi:L^0(\PB)\rightarrow
  L^0(\QB)$ such that
  \begin{align}
    \label{eq:LatticeHomo1}
    & X\wedge Y=0\text{ in } L^0(\PB)\,\Rightarrow\, \pi(X)\wedge
    \pi(Y)\text{ in }L^0(\QB),\\
    \label{eq:LatticeHomoOrderConti}
    &X_\alpha\downarrow 0 \text{ in }L^0(\PB)\,\Rightarrow\, \pi(X_\alpha)\downarrow 0\text{ in }L^0(\QB);\\
    \label{eq:PiConti1}
    &\begin{cases} %
      \xi_n,\xi,\,\eta\in L^0(\QB),\, |\xi_n|\leq |\eta|
      \text{ in }L^0(\QB) \,(\forall n),\, \xi_n\rightarrow \xi,\,\QB\text{-a.s. }\\
      \Rightarrow\,\exists X_n,X,Y\in L^0\text{ such that
      }\xi_n=\pi(X_n),\, \xi=\pi(X),\,\eta=\pi(Y),\\
      \qquad|X_n|\leq |Y|\text{ in }L^0\text{ and }X_n\rightarrow
      X,\,\PB\text{-a.s.}
    \end{cases}
  \end{align}
\end{lemma}
In general, a linear map from a Riesz space $E$ to another Riesz space
$F$ satisfying (\ref{eq:LatticeHomo1}) is called a \emph{lattice
  homomorphism}. (\ref{eq:LatticeHomoOrderConti}) says that $\pi$ is
\emph{order-continuous}, and such a lattice homomorphism is called a
\emph{normal homomorphism}. See \citep{aliprantis_Burkinshaw03} for
more information.

\begin{proof}[Proof of Lemma~\ref{lem:Sensitive1}]
  For each $X\in L^0$, let $\pi(X)$ be the $\QB$-equivalence class
  generated by a representative of $X$. This definition makes sense
  and does not depend on the choice of representative.  Indeed, if $f$
  and $g$ are two representatives of $X\in L^0$, then $f=g$
  $\PB$-a.s. by definition, hence $f=g$ $\QB$-a.s. since
  $\QB\ll\PB$. Thus the $\QB$-equivalence classes generated by $f$ and
  that by $g$ coincide. It is clear that $\pi:L^0\rightarrow L^0(\QB)$
  is linear and onto. To see (\ref{eq:LatticeHomo1}), suppose $X, Y\in
  L^0$ and $X\wedge Y=0$ in $L^0$. Then by definition, for any
  representatives $f\in X$ and $g\in Y$, we have $f\geq 0$ and $g\geq
  0$ $\PB$-a.s., hence $\QB$-a.s., and consequently $\pi(X)\geq 0$ and
  $\pi(Y)\geq 0$. Next, if $\xi\in L^0(\QB)$ and if $\xi\leq \pi(X)$,
  $\xi\leq \pi(Y)$ in $L^0(\QB)$, then taking a representative $h\in
  \xi$ in $L^0(\QB)$ with $f,g$ being same as above, we have $h\leq f$
  and $h\leq g$ $\QB$-a.s. Then putting $A=\{h\leq f,\,h\leq g\}$, we
  still have $h\mathds1_A\in \xi$ (since $\QB(A)=1$), and
  $h\mathds1_A\leq f$ and $h\mathds1_A\leq g$ $\PB$-a.s. (since $f\geq
  0$, $g\geq 0$ $\PB$-a.s.). Thus $h\mathds1_A\leq 0$ $\PB$-a.s.,
  hence $\QB$-a.s. Consequently, $\xi\leq 0$ in $L^0(\QB)$ which shows
  that $\pi(X)\wedge \pi(Y)=0$ in $L^0(\QB)$.

  The first line of (\ref{eq:PiConti1}) means that for some (hence
  all) representatives $f_n\in\xi_n$, $f\in \xi$ and $g\in \eta$,
  $|f_n|\leq |g|$ $\QB$-a.s. for all $n$, and $f_n\rightarrow f$
  $\QB$-a.s. Then putting $A=\{|f_n|\leq |g|\, (\forall n),\,
  f_n\rightarrow f\}\in \FC$, we see that $f_n\mathds1_A\in \xi_n$,
  $f\mathds1_A\in \xi$ and $g\mathds1_A\in \eta$ since $\QB(A)=1$,
  while $|f_n\mathds1_A|\leq |g\mathds1_A|$, $f_n\mathds1_A\rightarrow
  f\mathds1_A$ (pointwise). Hence if $X_n$ (resp. $X$, $Y$) denotes
  the $\PB$-class generated by $f_n\mathds1_A$ (resp. $f\mathds1_A$,
  $g\mathds1_A$), we have that $\xi_n=\pi(X_n)$, $\xi=\pi(X)$ and
  $\eta=\pi(Y)$ on the one hand, and on the other hand, $|X_n|\leq
  |Y|$ in $L^0$ and $X_n\rightarrow X$ $\PB$-a.s.

  Finally, for an onto lattice homomorphism $\pi$ to satisfy
  (\ref{eq:LatticeHomoOrderConti}), it is necessary and sufficient
  that the kernel of $\pi$ is a band (order-closed solid subspace) in
  $L^0$. In our case, the kernel of $\pi$ is $\NC_\PB(\QB)$ given by
  (\ref{eq:NPQ}), which is clearly a solid subspace of $L^0$. To prove
  the order-closedness, it suffices to check that for any increasing
  net $(X_\alpha)_\alpha\subset \NC_\PB(\QB)$ with $0\leq
  X_\alpha\uparrow X$ in $L^0$, we have $X\in\NC_\PB(\QB)$. But since
  $L^0$ has the countable sup property, there exists an increasing
  \emph{sequence} of indices $(\alpha_n)_n$ such that
  $X_{\alpha_n}\uparrow X$. Then the monotone convergence theorem
  shows that $\EB[X\Zh]=\lim_n\EB[X_{\alpha_n}\Zh]=0$, which implies
  $X=0$, $\QB$-a.s.
\end{proof}

\begin{remark}
  \label{rem:LatticeHomo}
  Taking $\eta=\sup_n|\xi_n|\in L^0(\QB)$ (if $\xi_n\rightarrow \xi$)
  (resp. $\xi_n\equiv \xi$, $\forall n$) in (\ref{eq:PiConti1}), we
  have also
    \begin{align}\label{eq:PiConti2}
      &\begin{cases}%
        & \xi_n,\xi\in L^0(\QB),\,\xi_n\rightarrow
        \xi,\QB\text{-a.s. }\\
        &\Rightarrow\,\exists X_n,X\in L^0,\,
        \xi_n=\pi(X_n),\,\xi=\pi(X),\,X_n\rightarrow
        X,\,\PB\text{-a.s.}
      \end{cases}
      \\\label{eq:PiSolid} & |\xi|\leq |\eta|\text{ in
      }L^0(\QB)\,\Rightarrow\,\exists X,Y\in L^0,\,
      \xi=\pi(X),\,\eta=\pi(Y),\,|X|\leq |Y|\text{ in }L^0.
    \end{align}
\end{remark}

Since $\pi: L^0\rightarrow L^0(\QB)$ is an onto lattice homomorphism,
we have $|\pi(X)|=\pi(|X|)$, and for any solid subspace $\Xs\subset
L^0$, the image $\Xs(\QB):=\pi(\Xs)$ is a solid subspace of $L^0(\QB)$
(see \citep[][Theorem~1.33]{aliprantis_Burkinshaw03}). If in addition
$\NC_\PB(\QB)\subset \Xs$, we see that $\pi(X)\in \Xs(\QB)$ if
\emph{and only if} $X\in \Xs$ (the if part is always true by
definition). Indeed, $\pi(X)\in \Xs(\QB)$ means $\pi(X)=\pi(X')$ with
$X'\in \Xs$, and then $\pi(X-X')=0$ in $L^0(\QB)$ $\Leftrightarrow$
$X-X'\in \NC_\PB(\QB)$, hence $X=X'+(X-X')\in \Xs+\NC_\PB(\QB)=\Xs$.
Noting that $\NC_\PB(\QB)\subset M^{\hat\varphi}_u$ by definition
(\ref{eq:NPQ}), the following three are all solid subspaces of
$L^0(\QB)$ of this type:
\begin{align*}
  M^{\hat\varphi}_u(\QB):=\pi(M^{\hat\varphi}_u),\quad
  L^{\hat\varphi}(\QB):=\pi(L^{\hat\varphi}),\quad
  M^{\hat\varphi}(\QB):=\pi(M^{\hat\varphi})
\end{align*}

 By (\ref{eq:ProbQ}) and $\|X\|_{\hat\varphi}=0$
  $\Leftrightarrow$ $\hat\varphi(\alpha|X|)=0$, $\forall \alpha>0$,
  the following is well-defined:
\begin{align}\label{eq:LatticeNormQ}
  \|\xi\|_{\hat\varphi,\QB}:=\|X\|_{\hat\varphi}\text{ if } \xi=
  \pi(X)\in L^0(\QB).
\end{align}
Note that $\|\xi\|_{\hat\varphi,\QB}<\infty$ iff $\xi\in
L^{\hat\varphi}(\QB)$ and $\|\xi\|_{\hat\varphi,\QB}=0$ if \emph{and
  only if} $\xi=0$ in $L^0(\QB)$ by construction. Thus
$\|\cdot\|_{\hat\varphi,\QB}$ is a lattice \emph{norm} on
$L^{\hat\varphi}(\QB)$ (hence on $M^{\hat\varphi}_u(\QB)$ and
$M^{\hat\varphi}(\QB)$ as well). The goal of this subsection is the
following:

\begin{theorem}
  \label{thm:MphiUOrderContiBanach}
  $(M^{\hat\varphi}_u(\QB),\|\cdot\|_{\hat\varphi,\QB})$ is an order
  continuous Banach lattice, i.e., $M^{\hat\varphi}_u(\QB)$ is
  complete for $\|\cdot\|_{\hat\varphi,\QB}$ and the norm
  $\|\cdot\|_{\hat\varphi,\QB}$ is order-continuous w.r.t.  the
  $\QB$-a.s. order:
  \begin{equation}
    \label{eq:LatticeNormQLebesgue} %
    |\xi_n|\leq \eta\in M_u^{\hat\varphi}(\QB), \,\xi_n\rightarrow
    0,\, \QB\text{-a.s. } \Rightarrow \, \|\xi_n\|_{\hat\varphi,\QB}
    \rightarrow 0.
  \end{equation}

\end{theorem}
On this occasion, we shall prove also the following at once:
\begin{proposition}
  \label{prop:BanachLatticeL}
  $L^{\hat\varphi}(\QB)$ is a Banach lattice for the lattice norm
  $\|\cdot\|_{\hat\varphi,\QB}$ and $M^{\hat\varphi}(\QB)$ is its
  closed linear subspace (hence itself a Banach lattice).
\end{proposition}

\begin{lemma}
  \label{lem:LatticeNormQ}
  $\|\cdot\|_{\hat\varphi,\QB}:L^0(\QB)\rightarrow \RB\cup\{+\infty\}$
  satisfies the following:
  \begin{align}
    \label{eq:LatticeNormQLSC}
    &\xi_n,\xi\in L^0(\QB),\,\xi_n\rightarrow
    \xi,\,\QB\text{-a.s. }\Rightarrow\, \|\xi\|_{\hat\varphi,\QB}\leq
    \liminf_n\|\xi_n\|_{\hat\varphi,\QB};
    \\ \label{eq:LatticeNormEmbedding}
    & \|\xi\|_{L^1(\QB)}\leq c_\QB\|\xi\|_{\hat\varphi,\QB},\forall
    \xi\in L^0(\QB)\text{ where } c_\QB=2\EB[\Zh].
  \end{align}
\end{lemma}
\begin{proof}
  If $\xi_n,\xi\in L^0(\QB)$, $\xi_n\rightarrow \xi$, $\QB$-a.s., then
  by (\ref{eq:PiConti2}), there exist $X_n,X\in L^0$ such that
  $\xi_n=\pi(X_n)$, $\xi=\pi(X)$ and $X_n\rightarrow X$,
  $\PB$-a.s. Thus by (\ref{eq:LatticeNormQ}) and (\ref{eq:NormLSC}),
  $\|\xi\|_{\hat\varphi,\QB}=\|X\|_{\hat\varphi}\leq
  \liminf_n\|X_n\|_{\hat\varphi}=\liminf_n\|\xi_n\|_{\hat\varphi,\QB}$,
  and we have (\ref{eq:LatticeNormQLSC}).  For
  (\ref{eq:LatticeNormEmbedding}), Lemma~\ref{lem:SeminormEmbedding}
  tells us that for each $\xi=\pi(X)\in L^0(\QB)$,
  $\cH^{-1}\|\xi\|_{L^1(\QB)}=\EB[X\Zh]\leq
  (1+\varphi^*_0(\Zh))\|X\|_{\hat\varphi}=
  (1+\varphi^*_0(\Zh))\|\xi\|_{\hat\varphi,\QB}$, thus we have
  (\ref{eq:LatticeNormEmbedding}) since $\varphi^*_0(\Zh)\leq 1$.
\end{proof}

\begin{proof}[Proof of Proposition~\ref{prop:BanachLatticeL} and
  Theorem~\ref{thm:MphiUOrderContiBanach}]
  We already know that
  $(L^{\hat\varphi}(\QB),\|\cdot\|_{\hat\varphi,\QB})$ is a normed
  Riesz space, and $M^{\hat\varphi}_u(\QB)$ and $M^{\hat\varphi}(\QB)$
  are its solid vector subspaces. To see that $L^{\hat\varphi}(\QB)$
  is complete, let $(\xi_n)_n\in L^{\hat\varphi}(\QB)$ be a Cauchy
  sequence for $\|\cdot\|_{\hat\varphi,\QB}$. Then by
  (\ref{eq:LatticeNormEmbedding}), it is also Cauchy in $L^1(\QB)$,
  hence admits the $\|\cdot\|_{L^1(\QB)}$-limit $\xi$ in $L^1(\QB)$, and we can choose a
  subsequence $(\xi_{n_k})_k$ so that $\xi_{n_k}\rightarrow \xi$,
  $\QB$-a.s. Then (\ref{eq:LatticeNormQLSC}) shows that
  $\|\xi-\xi_m\|_{\hat\varphi,\QB}\leq
  \liminf_k\|\xi_{n_k}-\xi_m\|_{\hat\varphi,\QB}$ for all $m$. Since
  the original sequence is Cauchy for $\|\cdot\|_{\hat\varphi,\QB}$,
  this shows that $\|\xi\|_{\hat\varphi,\QB}<\infty$ (hence $\xi\in
  L^{\hat\varphi}(\QB)$) and
  $\|\xi-\xi_n\|_{\hat\varphi,\QB}\rightarrow 0$.

  Suppose in addition that each $\xi_n$ belongs to
  $M^{\hat\varphi}(\QB)$, and write $\xi_n=\pi(X_n)$ with $X_n\in
  M^{\hat\varphi}$ and $\xi=\pi(X)$ with $X\in L^{\hat\varphi}$. Then
  for all $\alpha>0$, there is some large $n$ so that
  $\|X-X_n\|_{\hat\varphi} =\|\xi-\xi_n\|_{\hat\varphi,\QB}<
  1/2\alpha$ for which $\hat\varphi(2\alpha|X-X_n|)\leq 1$ , hence
  $\hat\varphi(\alpha|X|)\leq \frac12\hat\varphi(2\alpha|X-X_n|)
  +\frac12\hat\varphi(2\alpha|X_n|)\leq
  1+\frac12\hat\varphi(2\alpha|X_n|)$.  Consequently, $X\in
  M^{\hat\varphi}$, thus $\xi=\pi(X)\in M^{\hat\varphi}(\QB)$, and we
  deduce that $M^{\hat\varphi}(\QB)$ is closed in
  $L^{\hat\varphi}(\QB)$.

  For Theorem~\ref{thm:MphiUOrderContiBanach}, it remains to show that
  $M^{\hat\varphi}_u(\QB)$ is closed in $L^{\hat\varphi}(\QB)$, and
  $\|\cdot\|_{\hat\varphi,\QB}$ is order-continuous for the
  $\QB$-a.s. order (i.e., (\ref{eq:LatticeNormQLebesgue})). For the
  closedness, let $(\xi_n)_n$ and $\xi$ be as in the first paragraph
  and suppose that $\xi_n\in M^{\hat\varphi}_u(\QB)$ for each
  $n$. Then $\xi_n=\pi(X_n)$ with $X_n\in M^{\hat\varphi}_u$ for each
  $n$, and $\xi=\pi(X)$ with $X\in L^{\hat\varphi}$. Observe that
  \begin{align*}
    \|X\mathds1_{\{|X|>N\}}\|_{\hat\varphi} \leq
    \|X-X_n\|_{\hat\varphi} +\|X_n\mathds1_{\{|X|>N\}}\|_{\hat\varphi}
    =\|\xi- \xi_n\|_{\hat\varphi,\QB}
    +\|X_n\mathds1_{\{|X|>N\}}\|_{\hat\varphi}.
  \end{align*}
  The first term in the right hand side tends to $0$ as $n\rightarrow
  \infty$, while for each $n$, the second term tends to $0$ as
  $N\rightarrow\infty$ since $X\in M^{\hat\varphi}_u$. Taking a
  diagonal, we see that $X\in M^{\hat\varphi}_u$, hence $\xi=\pi(X)\in
  M^{\hat\varphi}_u(\QB)$. Therefore, $M^{\hat\varphi}_u(\QB)$ is
  closed.

  Finally, we show (\ref{eq:LatticeNormQLebesgue}). Let
  $(\xi_n)_n\subset M^{\hat\varphi}_u(\QB)$, $|\xi_n|\leq \eta\in
  M^{\hat\varphi}_u(\QB)$ and $\xi_n\rightarrow 0$ $\QB$-a.s. Then by
  (\ref{eq:PiConti1}), we can choose $X_n,Y$ with $\xi_n=\pi(X_n)$,
  $\eta=\pi(Y)$ (hence $Y\in M^{\hat\varphi}_u$), $|X_n|\leq |Y|$ and
  $X_n\rightarrow 0$ $\PB$-a.s. Then (\ref{eq:OrderContiSeminorm1})
  and (\ref{eq:LatticeNormQ}) show that
  $\|\xi_n\|_{\hat\varphi,\QB}=\|X_n\|_{\hat\varphi}\rightarrow 0$,
  and we deduce (\ref{eq:LatticeNormQLebesgue}).
\end{proof}

\begin{remark}[Sensitivity]
  \label{rem:OrderContiBanach}
  In general, $\QB$ is only absolutely continuous with respect to the
  original reference measure $\PB$ (not equivalent). From
  (\ref{eq:ProbQ}), a necessary and sufficient condition for the
  possibility of choosing an equivalent $\QB$ ($\QB\sim\PB$) is that
  \begin{equation}
    \label{eq:Sensitive}
    \forall A\in \FC\text{ with }\PB(A)>0,\, 
    \exists\varepsilon>0,\, \varphi_0(\varepsilon\mathds1_A)>0.
  \end{equation}
  In financial mathematics, this condition is called the
  \emph{sensitivity} of $\varphi_0$. See \citep[][Ch.~4]{MR2779313}
  for more information.

\end{remark}
\begin{corollary}
  \label{cor:BanachLatticeSensitive}
  If $\varphi_0$ is sensitive in the sense of (\ref{eq:Sensitive}),
  $(M^{\hat\varphi}_u,\|\cdot\|_{\hat\varphi})$ itself is an order
  continuous Banach lattice.
\end{corollary}

\subsection{Proof of Theorem~\ref{thm:MainLebExt1}}
\label{sec:ProofThMax}

We now proceed to Theorem~\ref{thm:MainLebExt1}. Recall that
$\NC_\PB(\QB)\subset M^{\hat\varphi}_u\subset \DC_0\cap (-\DC_0)$
where $\DC_0$ is defined by (\ref{eq:DX}). Thus if $X\in \DC_0$ and
$Y=X$ $\QB$-a.s. ($\Leftrightarrow$ $Y-X\in \NC_\PB(\QB)$), we have
$Y=X+(Y-X)\in \DC_0+\DC_0=\DC_0$ since $\DC_0$ is a convex cone. In
this case, we have also that $\hat\varphi(X)=\hat\varphi(Y)$. Indeed,
$X=Y$ $\QB$-a.s. implies $\EB[|X-Y|Z]=\EB[|X-Y|Z\mathds1_{\{X\neq
  Y\}}]=0$ for all $Z\in\dom\varphi^*_0$ by (\ref{eq:HalmosSavage1}),
hence $\hat\varphi(X)
=\sup_{Z\in\dom\varphi^*_0}(\EB[XZ]-\varphi^*_0(Z))
=\sup_{Z\in\dom\varphi^*_0}(\EB[YZ]-\varphi^*_0(Z))
=\hat\varphi(Y)$. Therefore,
\begin{align}\label{eq:PhiHatQ}
  \hat\varphi_\QB(\xi):=\hat\varphi(X)\text{ if }\xi=\pi(X)\in
  \DC_0(\QB):=\pi(\DC_0)
\end{align}
is well-defined as a function on $\DC_0(\QB):=\pi(\DC_0)$, hence in
particular on $L^{\hat\varphi}(\QB)$, $M^{\hat\varphi}(\QB)$ and on
$M^{\hat\varphi}_u(\QB)$. $\hat\varphi_\QB$ is convex (resp. monotone)
since $\pi$ is linear and $\hat\varphi$ is convex (resp. both $\pi$
and $\hat\varphi$ are monotone), and is finite on
$M^{\hat\varphi}(\QB)$ (hence on $M^{\hat\varphi}_u(\QB)$ in
particular).

\begin{proof}[Proof of Theorem~\ref{thm:MainLebExt1}]
  Recall that any monotone convex function on a Banach lattice is
  norm-continuous on the interior of its effective domain by the
  extended Namioka-Klee theorem \citep[][Theorem~1]{MR2648595}. Thus
  $\hat\varphi_\QB:M^{\hat\varphi}_u(\QB)\rightarrow\RB$ is
  $\|\cdot\|_{\hat\varphi,\QB}$-continuous as a finite valued monotone
  convex function on a Banach lattice $M^{\hat\varphi}_u(\QB)$, while
  since $\|\cdot\|_{\hat\varphi,\QB}$ is $\QB$-order continuous in the
  sense of (\ref{eq:LatticeNormQLebesgue}) by
  Theorem~\ref{thm:MphiUOrderContiBanach}, we deduce that
  $\hat\varphi_\QB:M^{\hat\varphi}_u(\QB)\rightarrow \RB$ is
  $\QB$-order continuous. Thus recalling that
  $\hat\varphi=\hat\varphi_\QB\circ\pi$ and $\pi:L^0(\PB)\rightarrow
  L^0(\QB)$ is order continuous, we obtain that
  $\hat\varphi:M^{\hat\varphi}_u\rightarrow \RB$ is $\PB$-order
  continuous. Consequently, $(\hat\varphi,M^{\hat\varphi}_u)$ is
  indeed a Lebesgue extension of $\varphi_0$.

  If $(\varphi,\Xs)$ is another Lebesgue extension, we must have
  $\Xs\subset M^{\hat\varphi}_u$ by (\ref{eq:ElemObs2}), and for any
  $X\in \Xs\subset M^{\hat\varphi}_u$, the Lebesgue properties of
  $\hat\varphi$ and $\varphi$ on $\Xs$ and
  $\hat\varphi|_{L^\infty}=\varphi|_{L^\infty}$ show that
  $\hat\varphi(X)=\lim_n\hat\varphi(X\mathds1_{\{|X|\leq
    n\}})=\lim_n\varphi_0(X\mathds1_{\{|X|\leq
    n\}})=\lim_n\varphi(X\mathds1_{\{|X|\leq n\}})=\varphi(X)$. Thus
  we have $\varphi=\hat\varphi|_\Xs$.
\end{proof}
  
\begin{remark}
  The three Orlicz-type spaces $M^{\hat\varphi}_u(\QB)$,
  $M^{\hat\varphi}(\QB)$ and $L^{\hat\varphi}(\QB)$ are also expressed
  using $\hat\varphi_\QB$ in forms parallel to those of original
  spaces:
\begin{align*}
  L^{\hat\varphi}(\QB)&=\{\xi\in L^0(\QB):\, \hat\varphi_\QB(\alpha|\xi|)<\infty,\,\exists \alpha>0\},\\
  M^{\hat\varphi}(\QB)&=\{\xi\in L^0(\QB):\, \hat\varphi_\QB(\alpha|\xi|)<\infty,\,\forall \alpha>0\},\\
  M^{\hat\varphi}_u(\QB)&=\{\xi\in L^0(\QB):\, \lim_N
  \hat\varphi_\QB(\alpha|\xi|\mathds1_{\{|\xi|>N\}})=0,\,\forall
  \alpha>0\}.
\end{align*}
For the last identity, we note that $\pi(|X|\mathds1_{\{|X|>N\}})
=\pi(|X|)\pi(\mathds1_{\{|X|>N\}})=|\pi(X)|\mathds1_{\{|\pi(X)|>N\}}$
which is straightforward from the definition of $\pi$ in
Lemma~\ref{lem:Sensitive1}.

\end{remark}

\section{Proof of Theorem~\ref{thm:MuCharact}}
\label{sec:MvarphiU}

\begin{proof}[Proof of Theorem~\ref{thm:MuCharact}: (1) $\Rightarrow$ (2)]
  If $\{XZ:\,\varphi^*_0(Z)\leq c\}$ is not uniformly integrable,
  there exists $\varepsilon>0$ such that for any $n$, there exists
  $A_n\in \FC$ and $Z_n\in L^1$ with $\PB(A_n)\leq 1/n$ and
  $\varphi^*_0(Z_n)\leq c$ and $\EB[|X|Z_n\mathds1_{A_n}] >
  \varepsilon$.  But then $\varepsilon < \EB[|X|Z_n\mathds1_{A_n}]
  \leq (1+c)\|X\mathds1_{A_n}\|_{\hat\varphi}$ for all $n$ by
  Lemma~\ref{lem:SeminormEmbedding}, hence $X\not\in
  M^{\hat\varphi}_u$ by Lemma~\ref{lem:ChMphiUNorm1}.
\end{proof}

Recall that if $X\in M^{\hat\varphi}$ (or more generally
$L^{\hat\varphi}$), $XZ\in L^1$ for any $Z\in\dom\varphi^*_0$ by
(\ref{eq:SeminormEmbedding}).
\begin{lemma}
  \label{lem:ProofUI1}
  Let $U\in M^{\hat\varphi}$ and suppose that $\{UZ:\,
  \varphi^*_0(Z)\leq c\}$ is uniformly integrable for each $c$. Then
  $\Lambda_{\beta,U,Y}:=\{Z:\,Z\in \dom\varphi^*_0,\,
  \EB[UYZ]-\varphi^*_0(Z)\geq -\beta\}$ is weakly compact in $L^1$ for
  all $\beta\in\RB$ and $Y\in L^\infty$.
\end{lemma}
\begin{proof}
  Since $\Lambda_{\beta,U,Y}$ is convex, it suffices to show that it
  is norm-closed and uniformly integrable. For the latter, fix an
  arbitrary $Z_0\in \dom\varphi^*_0$, and observe that
  \begin{align*}
    \EB[UYZ]&\leq \EB[2\|Y\|_\infty|U|(Z/2)]\leq
    \EB\left[2\|Y\|_\infty|U|\left(\frac12Z+\frac12Z_0\right)\right]\\
    &\leq\hat\varphi(2\|Y\|_\infty|U|)+\frac12\varphi^*_0(Z)+\frac12\varphi^*_0(Z_0).
  \end{align*}
  Thus $Z\in \Lambda_{\beta,U,Y}$ implies that
  \begin{align*}
    - \beta&\leq \EB[UYZ]-\varphi^*_0(Z)
    \leq\hat\varphi(2\|Y\|_\infty|U|)+\frac12\varphi^*_0(Z_0)-\frac12\varphi^*_0(Z).
  \end{align*}
  Putting
  $\beta':=2\beta+2\hat\varphi(2\|Y\|_\infty|U|)+\varphi^*_0(Z_0)<\infty$
  (since $U\in M^{\hat\varphi}$), this tells us that
  $\Lambda_{\beta,U,Y}\subset \{Z:\,\varphi^*_0(Z)\leq \beta'\}$ and
  the latter set is uniformly integrable by the fundamental assumption
  that $\varphi_0$ is Lebesgue on $L^\infty$ and
  Theorem~\ref{thm:JSTLinfty}.

  To see the closedness, let $Z_n\in \Lambda_{\beta,U,Y}$ and
  $Z_n\rightarrow Z\in L^1$ in norm.  Then $Z_n\rightarrow Z$ in
  probability, hence $UYZ_n\rightarrow UYZ$ in probability as well. On
  the other hand, from what we just proved and the assumption of
  lemma, $(UZ_n)_n$ is uniformly integrable, thus so is $(UYZ_n)_n$
  since $Y\in L^\infty$. Consequently, $\EB[UYZ]=\lim_n\EB[UYZ_n]$,
  and since $\varphi^*_0$ is lower semicontinuous on $L^1$, we have
  also $\varphi^*_0(Z)\leq \liminf_n\varphi_0^*(Z_n)$. Summing up,
  \begin{align*}
    \EB[UYZ]-\varphi^*_0(Z)&\geq
    \lim_n\EB[UYZ_n]-\liminf_n\varphi^*_0(Z_n)\\
    &\geq \limsup_n(\EB[UYZ_n]-\varphi^*_0(Z_n))\geq -\beta.
  \end{align*}
  Hence $Z\in \Lambda_{\beta,U,Y}$.
\end{proof}

\begin{proof}[Proof of Theorem~\ref{thm:MuCharact}: (2) $\Rightarrow$
  (3) and (\ref{eq:MaxAttained1})]
  For $U\in M^{\hat\varphi}$, $Y\in L^\infty$, we put
  $l_{U,Y}(Z):=\EB[UYZ]-\varphi^*_0(Z)$. Then Lemma~\ref{lem:ProofUI1}
  tells us that if $\{UZ:\,\varphi^*_0(Z)\leq c\}$ is uniformly
  integrable for each $c>0$, $l_{U,Y}$ is weakly upper semicontinuous
  and all its upper level sets are weakly compact for each $Y\in
  L^\infty$, and thus $\sup_{Z\in \dom\varphi^*_0}l_{U,Y}(Z)$ is
  attained. By the condition (2) of Theorem~\ref{thm:MuCharact}, this
  applies to $U=X$ and $Y=1$ (constant), and we obtain
  (\ref{eq:MaxAttained1}). For (3), we note that $|X|\leq |X|\vee1\leq
  |X|+1$ and $\{Z:\, \varphi^*_0(Z)\leq c\}$ is uniformly integrable
  for each $c>0$ by Theorem~\ref{thm:JSTLinfty} and the Lebesgue
  property of $\varphi_0$ on $L^\infty$, hence (2) implies also that
  $\{(|X|\vee 1)Z:\, \varphi^*_0(Z)\leq c\}$ is uniformly integrable
  too. Therefore, the above argument applies to $U=|X|\vee 1\in
  M^{\hat\varphi}$ and arbitrary $Y\in L^\infty$, showing that the
  supremum $\sup_{Z\in \dom\varphi^*_0}(\EB[(|X|\vee
  1)YZ]-\varphi^*_0(Z))=\sup_{Z\in \dom\varphi^*_0}l_{|X|\vee 1,Y}(Z)$
  is attained for each $Y\in L^\infty$. This concludes the proof of
  (2) $\Rightarrow$ (3).
\end{proof}

\begin{proof}[Proof of Theorem~\ref{thm:MuCharact}: (2) $\Rightarrow$
  (1)]
  We apply a version of minimax theorem (Theorem~\ref{thm:AppMinimax})
  to the function $L^\infty\times\dom\varphi^*_0\ni (Y,Z)\mapsto
  f(Y,Z):= \EB[|X|YZ]-\varphi^*_0(Z)$. We already know under (2) that
  for each $Y\in L^\infty$, $Z\mapsto f(Y,Z)$ is concave, weakly upper
  semicontinuous on $\dom\varphi^*_0$ and all its level sets are
  weakly compact by Lemma~\ref{lem:ProofUI1} applied to $U=|X|$. On
  the other hand $Y\mapsto f(Y,Z)$ is affine (hence convex). Thus for
  any convex set $C\subset L^\infty$, we have
  \begin{align}\label{eq:ProofUIMinimax}
    \inf_{Y\in C}\sup_{Z\in
      \dom\varphi^*_0}\left(\EB[|X|YZ]-\varphi^*_0(Z)\right)=\sup_{Z\in\dom\varphi^*_0}\inf_{Y\in
      C}\left(\EB[|X|YZ]-\varphi^*_0(Z)\right).
  \end{align}
  Let $C_1$ be the convex hull $\mathrm{conv}(
  \mathds1_{\{|X|>N\}},N\in\NB)$. Observe that for any $n\in\NB$,
  $\lambda_i\geq 0$, $\lambda_1+\cdots+\lambda_n=1$ and
  $N_1<N_2<\cdots<N_n$, we have $\mathds1_{\{|X|>N_n\}}\leq
  \lambda_1\mathds1_{\{|X|>N_1\}}+\cdots+\lambda_n\mathds1_{\{|X|>N_n\}}\leq
  \mathds1_{\{|X|>N_1\}}$ and every element of $C_1$ is written in the
  form of middle expression. Thus for any $\alpha>0$,
  \begin{align*}
    \lim_N&\hat\varphi(\alpha|X|\mathds1_{\{|X|>N\}})=\inf_{Y\in\alpha
      C_1}\hat\varphi(|X|Y) =\inf_{Y\in\alpha C_1}\sup_{Z\in
      \dom\varphi^*_0}\left(\EB[|X|YZ]-\varphi^*_0(Z)\right)\\
    &\stackrel{\text{~(\ref{eq:ProofUIMinimax})}}=\sup_{Z\in
      \dom\varphi^*_0}\left(\inf_{Y\in \alpha
        C_1}\EB[|X|YZ]-\varphi^*_0(Z)\right)\\
    &=\sup_{Z\in \dom\varphi^*_0}\left(\lim_N\alpha
      \EB[|X|\mathds1_{\{|X|>N\}}Z]-\varphi^*_0(Z)\right)=\sup_{Z\in\dom\varphi^*_0}-\varphi^*_0(Z)=0.
  \end{align*}
  Thus $X\in M^{\hat\varphi}_u$.  
\end{proof}

We proceed to the implication (3) $\Rightarrow$ (2). This will follow
from the following version of \emph{perturbed James's theorem} recently
obtained by \citep{orihuelaRuizGalan12:_james}:
\begin{theorem}[\citep{orihuelaRuizGalan12:_james}, Theorem~2]
  \label{thm:James}
  Let $E$ be a real Banach space and
  $f:E\rightarrow\RB\cup\{+\infty\}$ be a function which is coercive,
  i.e.,
  \begin{equation}
    \label{eq:Coercive}
    \lim_{\|x\|\rightarrow\infty}\frac{f(x)}{\|x\|}=+\infty.
  \end{equation}
  Then if the supremum $\sup_{x\in E}(\langle x,x^*\rangle-f(x))$ is
  attained for every $x^*\in E^*$, the level set $\{x\in E:\, f(x)\leq
  c\}$ is relatively weakly compact for each $c\in\RB$.
\end{theorem}

We shall apply this theorem with $E=L^1$. We first make a ``change of
variable''.  For $U\in M^{\hat\varphi}$ with $U\geq 1$ a.s., we set
\begin{equation}
  \label{eq:GX}
  g_U(Z) :=\varphi^*_0(Z/U) 
  =\sup_{\xi\in L^\infty}\left(\EB[\xi Z/U]-\hat\varphi(\xi)\right),\quad \forall Z\in L^1.
\end{equation}
Note that $\dom g_U\subset L^1_+$ since $\varphi_0$ is monotone (see
(\ref{eq:MonotoneZ})), and that 
\begin{equation}
  \label{eq:GXDom}
  \begin{split}
    \{Z\in L^1:\, g_U(Z)\leq c\}&=\{UZ':\,Z'\in L^1,\, \varphi^*_0(Z')\leq c\},\\
    \dom g_U&=U\dom\varphi^*_0=\{UZ:\, Z\in \dom\varphi^*_0\}.
  \end{split}
\end{equation}
(Remember that $XZ\in L^1$ for any $X\in M^{\hat\varphi}$ and $Z\in
\dom\varphi^*_0$ by (\ref{eq:SeminormEmbedding}).)

\begin{lemma}
  \label{lem:Coercive}
  Let $U\in M^{\hat\varphi}$ with $U\geq 1$. Then $g_U$ is coercive:
  \begin{equation}
    \label{eq:UICoercive}
    \lim_{\|Z\|_1\rightarrow\infty}g_U(Z)/\|Z\|_1=\infty.
  \end{equation}

\end{lemma}
\begin{proof}
  For any $n$ and $\alpha>0$ (constant), $\alpha U\mathds1_{\{U\leq
    n\}}\in L^\infty$, hence from the definition of $g_U$,
  \begin{align*}
    g_U(Z)&\geq \EB[\alpha U\mathds1_{\{U\leq
      n\}}(Z/U)]-\hat\varphi(\alpha U\mathds1_{\{U\leq n\}})
    = \alpha\|Z\mathds1_{\{U\leq
      n\}}\|_1-\hat\varphi(\alpha U\mathds1_{\{U\leq n\}})\\
    &\rightarrow \alpha\|Z\|_1-\hat\varphi(\alpha U),\,\forall Z\in L^1_+,
  \end{align*}
  while $g_U(Z)=\infty$ if $Z\in L^1\setminus L^1_+$. Here the last
  convergence follows from $0\leq \alpha U\mathds1_{\{U\leq
    n\}}\uparrow \alpha U$, so $\hat\varphi(\alpha U)=
  \lim_n\hat\varphi(\alpha U\mathds1_{\{ U\leq n\}})$ by
  Lemma~\ref{lem:HatPhiMon1}. Since $\hat\varphi(\alpha U)<\infty$ for
  any $\alpha>0$ by $U\in M^{\hat\varphi}$, this shows
  (\ref{eq:UICoercive}).
\end{proof}

\begin{proof}[Proof of Theorem~\ref{thm:MuCharact}: (3) $\Rightarrow$
  (2)]
  Suppose (3), namely, for some $\varepsilon>0$, the supremum \linebreak
  $\sup_{Z\in \dom\varphi^*_0}(\EB[(|X|\vee
  \varepsilon)YZ]-\varphi^*_0(Z))=\sup_{Z\in
    \dom\varphi^*_0}(\EB[(|X/\varepsilon|\vee 1)(\varepsilon
  Y)Z]-\varphi^*_0(Z))$ is attained for every $Y\in L^\infty$. Putting
  $U=|X/\varepsilon|\vee 1\in M^{\hat\varphi}$, this says that for any
  $Y\in L^\infty$, there exists $Z_{U,Y}\in \dom\varphi^*_0\subset
  L^1$ such that $\hat\varphi(YU) =\EB[YUZ_{U,Y}]-\varphi^*_0(Z_{U,Y})
  =\EB[Y(UZ_{U,Y})]-g_U(UZ_{U,Y})$. On the other hand, for any $Z'\in
  \dom g_U=U\dom\varphi^*_0$,
  \begin{align*}
    \EB[YZ']-g_U(Z')&=\EB\left[YU\frac{Z'}{U}\right]-\varphi^*_0\left(\frac{Z'}{U}\right)\leq
    \hat\varphi(YU).
    % &=\EB[Y(|X|Z_{X,Y})]-g_X(|X|Z_{X,Y}),
  \end{align*}
  Thus $\sup_{Z\in L^1}(\EB[YZ]-g_U(Z))$ is attained for all $Y\in
  L^\infty$, hence Theorem~\ref{thm:James} shows that $\{Z'\in L^1:\,
  g_U(Z')\leq c\}=\{UZ:\, \varphi^*_0(Z)\leq c\}$ is relatively weakly
  compact ($\Leftrightarrow$ uniformly integrable) for each
  $c>0$. Since $|X|\leq \varepsilon U$, we deduce that $\{XZ:\,
  \varphi^*_0(Z)\leq c\}$ is uniformly integrable for each
  $c>0$.
\end{proof}

\section{Proof of Theorem~\ref{thm:JSTGeneral}}
\label{sec:JSTProof}

We use the notation of Theorem~\ref{thm:JSTGeneral}, namely,
$\psi:\Xs\rightarrow\RB$ is a finite monotone convex function with the
\emph{Fatou property} (\ref{eq:FatouX2}) on a solid space $\Xs\subset
L^0$ containing the constants, and we put
$\psi_\infty:=\psi|_{L^\infty}$, $\psi^*_\infty(Z)=\sup_{X\in
  L^\infty}(\EB[XZ]-\psi(X))=(\psi|_{L^\infty})^*(Z)$ and
\begin{equation*}
%  \label{eq:PsiHat}
  \hat\psi(X)=\sup_{Z\in\dom\psi^*\infty}(\EB[XZ]-\psi_\infty^*(Z)),
\end{equation*}
on $\DC_0=\{X\in L^0:\, X^-Z\in L^1,\,\forall Z\in
\dom\psi^*_\infty\}$.  Remember that we do not \emph{a priori} assume
the Lebesgue property of $\psi_\infty$ on $L^\infty$ here, but it is
implied by any of conditions (1) - (4) of Theorem~\ref{thm:JSTGeneral}
as we shall see in the proof below.  Note also that we can \emph{and
  do in the sequel} assume that $\psi(0)=0$, replacing $\psi$ by
$\psi-\psi(0)$.

\begin{proof}[Proof of Theorem~\ref{thm:JSTGeneral}: (1) $\Rightarrow$
  (2)]
  If $\psi$ is finite and has the Lebesgue property on $\Xs$,
  $\psi_\infty$ is a finite monotone convex function with the Lebesgue
  property on $L^\infty$. Thus Theorem~\ref{thm:MainLebExt1} applies
  to $\varphi_0=\psi_\infty$ (hence $\hat\varphi=\hat\psi$) implying
  that $(\hat\psi,M^{\hat\psi}_u)$ is the maximum Lebesgue extension
  of $\psi_\infty$. On the other hand, $(\psi,\Xs)$ is another
  Lebesgue extension of $\psi_\infty$, hence we must have $\Xs\subset
  M^{\hat\varphi}_u$. Consequently, (2) follows from
  Theorem~\ref{thm:MuCharact} ((1) $\Rightarrow$ (2)).
\end{proof}
\begin{proof}[Proof of Theorem~\ref{thm:JSTGeneral}: (2) $\Rightarrow$
  (3)]
  Since $\psi$ is supposed to have the Fatou property on $\Xs$,
  $\psi_\infty=\psi|_{L^\infty}$ has the Fatou property on
  $L^\infty$. Then condition (2) of Theorem~\ref{thm:JSTGeneral}
  applied to $X=1$ implies through Theorem~\ref{thm:JSTLinfty} that
  $\psi_\infty$ has the Lebesgue property on $L^\infty$. On the other
  hand, since $\psi$ is finite on $\Xs$ and has the Fatou property
  ($\Leftrightarrow$ continuous from below), we see that
  $\hat\psi(\alpha |X|)=\lim_n\hat\psi(\alpha |X|\wedge
  n)=\lim_n\psi(\alpha |X|\wedge n)=\psi(\alpha|X|)<\infty$, hence
  $\Xs\subset M^{\hat\psi}$.  Consequently, for each $X\in \Xs$, the
  assumption of Theorem~\ref{thm:MuCharact} ((2) $\Rightarrow$
  (\ref{eq:MaxAttained1})) is satisfied with $\varphi_0=\psi_\infty$
  ($\Rightarrow$ $\hat\varphi=\hat\psi$), thus the supremum
  $\sup_{Z\in \dom\psi^*_\infty}(\EB[XZ]-\psi^*_\infty(Z))$ is
  attained.
\end{proof}

The implication (3) $\Rightarrow$ (1) is a little more subtle.  We
first note that condition (3) of Theorem~\ref{thm:JSTGeneral}
restricted to $L^\infty\subset \Xs$ again implies the Lebesgue
property of $\psi_\infty=\psi|_{L^\infty}$ on $L^\infty$. Thus
$\varphi_0=\psi_\infty$ satisfies our standing assumption
(Assumption~\ref{as:Standing}). Let $\QB\ll \PB$ be the probability
measure constructed in Lemma~\ref{lem:SensitivityMaximal1} with
$\varphi_0=\psi_\infty$, i.e., a measure such that $\QB(A)=0$ iff
$\psi_\infty(\alpha\mathds1_A)=0$ for all $\alpha>0$, and we use the
notation (adapted to $\varphi_0=\psi_\infty$, $\hat\varphi=\hat\psi$)
of Section~\ref{sec:ChMeas}: $\pi:L^0\rightarrow L^0(\QB)$ (the
order-continuous lattice homomorphism constructed in
Lemma~\ref{lem:Sensitive1}),
$M^{\hat\psi}_u(\QB)=\pi(M^{\hat\psi}_u)$,
$M^{\hat\psi}(\QB)=\pi(M^{\hat\psi})$ and $\hat\psi_\QB$ (defined by
(\ref{eq:PhiHatQ}) with $\hat\varphi=\hat\psi$). Then
$\Xs(\QB):=\pi(\Xs)$ is a solid subspace of $L^0(\QB)$.

\begin{lemma}
  \label{lem:JSTsensitive2}
  With the notation above and the condition (3) of
  Theorem~\ref{thm:JSTGeneral},
  \begin{equation}
    \label{eq:JSTSensitiveFatou}
    X,Y\in \Xs,\,X=Y,\,\QB\text{-a.s. }\Rightarrow\, \psi(X)=\psi(Y).
  \end{equation}
  In particular,
  \begin{equation}
    \label{eq:PsiQ}
    \psi_\QB(\xi):=\psi(X),\, \xi=\pi(X)\in \Xs(\QB)
  \end{equation}
  is well defined as a monotone convex function on $\Xs(\QB)$, and it
  has the $\QB$-Fatou property on $\Xs(\QB)$, and thus
  $\psi_\QB(\xi)=\hat\psi_\QB(\xi)$ for all $\xi\in\Xs_+(\QB)$.

\end{lemma}
\begin{proof}
  We first claim that for any $X,Y\in\Xs$,
  \begin{equation}
    \label{eq:JSTLemmaSensitive1}
    \psi(\alpha|X-Y|)=0,\,\forall \alpha>0\,\Rightarrow \, \psi(X)=\psi(Y).
  \end{equation}
  To see this, we note that
  \begin{align*}
    \psi(X)-\psi(Y)\leq
    \frac1\alpha\psi(\alpha|X-Y|)+\frac{\alpha-1}\alpha\psi\left(\frac\alpha{\alpha-1}Y\right)-\psi(Y),\quad\forall
    \alpha>1.
  \end{align*}
  Since $\psi$ is finite, the finite convex function
  $\beta\mapsto\psi(\beta Y)$ is continuous on $\RB$, thus $f(\beta)=
  \psi(\beta Y)/\beta$ is continuous at $\beta=1$ with
  $f(1)=\psi(Y)$. Therefore, for any $\varepsilon>0$, there exists
  $\alpha_\varepsilon>1$ so that
  $\frac{\alpha_\varepsilon-1}{\alpha_\varepsilon}
  \psi\left(\frac{\alpha_\varepsilon}{\alpha_\varepsilon-1}Y\right)-\psi(Y)
  <\varepsilon$. Combining this with the assumption
  $\psi(\alpha|X-Y|)=0$ for all $\alpha$, we see that
  $\psi(X)-\psi(Y)<\varepsilon$ for all $\varepsilon>0$, hence
  $\psi(X)\geq\psi(Y)$. Changing the roles of $X$ and $Y$, we have
  also $\psi(X)\leq \psi(Y)$, and (\ref{eq:JSTLemmaSensitive1})
  follows.

  If $X=Y$, $\QB$-a.s., then by the construction of $\QB$ (with
  $\varphi_0=\psi_\infty$), we see that $\psi(\alpha|X-Y|\wedge
  n)=\psi_\infty(\alpha|X-Y|\wedge n)=0$ for all $n$, then the Fatou
  property of $\psi$ implies $\psi(\alpha|X-Y|)\leq
  \liminf_n\psi(\alpha|X-Y|\wedge n)=0$. Thus
  (\ref{eq:JSTSensitiveFatou}) follows from
  (\ref{eq:JSTLemmaSensitive1}).

  It is clear from (\ref{eq:JSTSensitiveFatou}) that $\psi_\QB$ of
  (\ref{eq:PsiQ}) is well-defined and finite on $\Xs(\QB)$. To see the
  $\QB$-Fatou property, suppose $|\xi_n|\leq |\eta|$ ($\forall n$) for
  some $\eta\in\Xs(\QB)$ and $\xi_n\rightarrow \xi$ $\QB$-a.s.. Then
  by (\ref{eq:PiConti1}), we can choose $X_n, X\in L^0$ and $Y\in\Xs$
  so that $\xi_n=\pi(X_n)$, $\xi=\pi(X)$, $\eta=\pi(Y)$ with
  $|X_n|\leq |Y|$ in $L^0$ (hence $X_n,X\in \Xs$ by the solidness) and
  that $X_n\rightarrow X$ $\PB$-a.s. Then the $\PB$-Fatou property of
  the original $\psi$ shows that $\psi_\QB(\xi)=\psi(X)\leq
  \liminf_n\psi(X_n)=\liminf_n\psi_n(\xi_n)$.  The final assertion
  follows since if $\xi\geq 0$, then $\QB$-Fatou property shows
  $\psi_\QB(\xi)=\lim_n\psi_\QB(\xi\wedge n)=\lim_n\hat\psi(\xi\wedge
  n)=\hat\psi_\QB(\xi)$.
\end{proof}

Consequently, we have $\psi=\psi_\QB\circ\pi$ and recall that
$\pi:L^0\rightarrow L^0(\QB)$ is order-continuous. Thus $\psi$ is
order-continuous on $\Xs$ as soon as $\psi_\QB$ is $\QB$-order
continuous on $\Xs(\QB)=\pi(\Xs)$ which is a solid subspace of
$M^{\hat\psi}_u(\QB)$. Then if $\Xs(\QB)$ was further norm-closed in
$M^{\hat\psi}_u(\QB)$, we could conclude that
$(\Xs(\QB),\|\cdot\|_{\hat\psi,\QB})$ is an order-continuous Banach
lattice on its own right, hence any finite monotone convex function on
it is order continuous.  But there is no guarantee that $\Xs(\QB)$ is
closed in $M^{\hat\varphi}_u(\QB)$, so we need a trick.

\begin{lemma}
  \label{lem:JST3}
  In addition to the assumption of Lemma~\ref{lem:JSTsensitive2}, we
  suppose that $\Xs\subset M^{\hat\psi}_u$. Then $\psi$ has the
  Lebesgue property on $\Xs$, hence a fortiori $\psi=\hat\psi|_\Xs$.

\end{lemma}
\begin{proof}
  To see the Lebesgue property of $\psi$ on $\Xs$, it suffices to show
  that $\psi_\QB$ has the $\QB$-Lebesgue property on $\Xs(\QB)$, and
  for the latter, we have to show that for any $\eta\in\Xs(\QB)$,
  \begin{equation}
    \label{eq:JSTLemmaOrderConti2}
    |\xi_n|\leq |\eta|\,(\forall
    n),\,\xi_n\rightarrow \xi\in \Xs(\QB)\,\QB\text{-a.s. }
    \Rightarrow\, \psi_\QB(\xi)=\lim_n\psi_\QB(\xi_n).
  \end{equation}
  Thus in the sequel, we fix an $\eta=\pi(Y)\in \Xs(\QB)$, and note
  that $\Xs(\QB)$ is solid subspace of $M^{\hat\psi}_u(\QB)$ since
  $\Xs$ is a solid subspace of $M^{\hat\psi}_u$ and $\pi$ is an onto
  lattice homomorphism.

  \noindent\textbf{Step~1.} Define
  \begin{equation}
    \label{eq:PrincipalBand1}
    B_\eta(\QB) :=\{\zeta\in M^{\hat\psi}_u(\QB):\, |\zeta|\wedge n|\eta|\uparrow |\zeta|\}.
  \end{equation}
  This is the principal band generated by $\eta$ in
  $M^{\hat\psi}_u(\QB)$, i.e., it is the smallest \emph{order closed}
  solid subspace (band) of $M^{\hat\psi}_u(\QB)$ containing
  $\eta$. Consequently, $B_\eta(\QB)$ is \emph{norm closed}
  (\citep[][Theorem~8.43]{aliprantis_border06}) in the
  order-continuous Banach lattice
  $(M^{\hat\psi}_u(\QB),\|\cdot\|_{\hat\psi,\QB})$, so
  $(B_\eta(\QB),\|\cdot\|_{\hat\psi,\QB})$ is itself an
  order-continuous Banach lattice. Hence the extended Namioka-Klee
  theorem shows that any finite monotone convex function on
  $B_\eta(\QB)$ is order-continuous.

  \noindent\textbf{Step~2.} 
  Define
  \begin{equation}
    \label{eq:BandExtension}
    \psi^\eta_{\QB}(\xi):=\lim_m\lim_n\psi_\QB((\xi\vee (-m|\eta|)\wedge n|\eta|),\quad \xi\in B_\eta(\QB).
  \end{equation}
  Observe that $(\xi\vee (-m|\eta|))\wedge n|\eta|\in \Xs(\QB)$ for
  each $m,n$ since $\eta\in \Xs(\QB)$ and $\Xs(\QB)$ is solid, hence
  $\psi^\eta_\QB$ is well-defined at least as a
  $[-\infty,\infty]$-valued monotone function, and it is
  straightforward to deduce from the monotonicity and convexity of
  $\psi_\QB$ that $\psi^\eta_\QB$ is also monotone and
  convex. Moreover, $\psi^\eta_\QB$ is finite on $B_\eta(\QB)$. To see
  this, note first that for all $\xi\in B_\eta(\QB)\subset
  M^{\hat\psi}_u(\QB)$, Lemma~\ref{lem:JSTsensitive2} shows that
\begin{align*}
  \psi^\eta_{\QB}(|\xi|)=\lim_n\psi_\QB(|\xi|\wedge n|\eta|)
  =\lim_n\hat\psi_\QB(|\xi|\wedge n|\eta|)=\hat\psi_\QB(|\xi|)<\infty.
\end{align*}
% since $X\in B_\eta(\QB)\subset M^{\hat\varphi}_u(\QB)$ by assumption.
On the other hand,
$\psi^\eta_{\QB}(-|\xi|)=\lim_n\psi_\QB(-(|\xi|\wedge n|\eta|))$ by
definition, and
\begin{align*}
  0&=2\psi_\QB(0)\leq \psi_\QB(|\xi|\wedge
  n|\eta|)+\psi_\QB(-(|\xi|\wedge n|\eta|))\\
  &\leq \hat\psi_\QB(|\xi|)+\psi_\QB(-(|\xi|\wedge n|\eta|)),\quad
  \forall n,
\end{align*}
hence $\psi^\eta_\QB(-|\xi|)=\inf_n\psi_\QB(-(|\xi|\wedge n|\eta|)\geq
-\hat\psi(-|\xi|)>-\infty$. Consequently, \textbf{Step~1} tells us
that $\psi^\eta_\QB$ is $\QB$-order continuous on $B_\eta(\QB)$ as a
finite monotone convex function on an order continuous Banach lattice.

\noindent\textbf{Step~3.} 
Though $B_\eta(\QB)$ may not contain the \emph{whole} $\Xs$, we see
that if $|\xi|\leq |\eta|$, then $\xi\in\Xs(\QB)\cap B_\eta(\QB)$ and
$(\xi\vee (-m|\eta|))\wedge n|\eta|=\xi$ for all $m,n$, hence
$\psi^\eta_\QB(\xi)=\psi_\QB(\xi)$. In particular, if $|\xi_n|\leq
|\eta|$ and $\xi_n\rightarrow \xi$, $\QB$-a.s., we have
$\psi_\QB(\xi_n)=\psi^\eta_\QB(\xi_n)
\rightarrow\psi^\eta_\QB(\xi)=\psi_\QB(\xi)$ by \textbf{Step~2}, and
we have (\ref{eq:JSTLemmaOrderConti2}).
\end{proof}

\begin{proof}[Proof of Theorem~\ref{thm:JSTGeneral}: (3) $\Rightarrow$
  (1) and (4)]
  Remember that (3) restricted to $L^\infty$ implies that
  $\psi_\infty=\psi|_{L^\infty}$ has the Lebesgue property on
  $L^\infty$.  Also, since $\Xs$ is supposed to be solid, we have
  $(|X|\vee 1)Y\in \Xs$ for all $X\in\Xs$ and $Y\in L^\infty$, hence
  the condition (3) of Theorem~\ref{thm:JSTGeneral} already implies
  that the supremum $\sup_{Z\in \dom\psi^*_\infty}(\EB[(|X|\vee
  1)YZ]-\psi^*_\infty(Z))$ is attained for any $X\in \Xs$, $Y\in
  L^\infty$ and $Z\in \dom\psi^*_\infty$. Hence we see from
  Theorem~\ref{thm:MuCharact} ((3) $\Rightarrow$ (1)) that $\Xs\subset
  M^{\hat\psi}_u$. Thus by Lemma~\ref{lem:JST3}, $\psi$ has the
  Lebesgue property on $\Xs$ (thus (1)), and
  Theorem~\ref{thm:MainLebExt1} shows that
  $\psi(X)=\hat\psi(X)=\sup_{Z\in
    \dom\psi^*_\infty}(\EB[XZ]-\psi^*_\infty(Z))$, hence we have (4)
  since the supremum is supposed to be attained.
\end{proof}

\section{Convex Risk Measures}
\label{sec:ConvRiskFunc}

Here we consider convex risk measures as our motivating class of
monotone convex functions.  In mathematical finance, a convex risk
measure on a solid space $\Xs$ is a proper convex function $\rho$
which is monotone \emph{decreasing} in the a.s. order and satisfies
the \emph{cash-invariance}: $\rho(X+c)=\rho(X)-c$ if $X\in \Xs$ and
$c\in\RB$. Making a change of sign, we call a proper monotone
(increasing) convex function $\varphi$ on $\Xs$ a \emph{convex risk
  function} if
\begin{equation}
  \label{eq:CashAdd}
  \varphi(X+c)=\varphi(X)+c,\,\forall X\in \Xs,\,c\in\RB.
\end{equation}
The relation between the two notions is obvious; if $\varphi$ is a
convex risk function, then $\rho(X)=\varphi(-X)$ is a convex risk
measure, and also $-\varphi(-X)$ is called a \emph{concave monetary
  utility function}. Though it is just a matter of notation, we prefer
monotone \emph{increasing} and \emph{convex} functions which fit to
our and standard notation of convex analysis, and it is also less
confusing. Also, a convex risk function $\varphi$ is called
\emph{coherent} if it is positively homogeneous: $\varphi(\alpha
X)=\alpha\varphi(X)$ if $\alpha\geq 0$. We refer the reader to
\citep[][Ch.~4]{MR2779313} for a comprehensive account.

When $\Xs=L^\infty$, condition (\ref{eq:CashAdd}) for a monotone
convex function $\varphi$ is equivalent to 
\begin{equation}
  \label{eq:PenaltyDomRiskFunc}
  \varphi^*(Z)<\infty\,\Rightarrow\, Z\geq 0\text{ and }\EB[Z]=1,
\end{equation}
i.e., $\varphi^*(Z)$ is finite only if $Z$ is a Radon-Nikodým density
of a probability measure, say $Q$, absolutely continuous w.r.t. $\PB$.
Adopting the usual convention of identifying a probability measure
$Q\ll\PB$ with its density $dQ/d\PB$, the representation
(\ref{eq:SigmaAddDual1}) is written as
\begin{equation}
  \label{eq:RobRepLinfty1}
  \varphi(X)=\sup_{Q\in\QC_\varphi}(\EB_Q[X]-\varphi^*(Q)),
  \, X\in L^\infty,
\end{equation}
where $\QC_{\varphi^*}:=\{Q\ll\PB:\, \text{probability, }dQ/d\PB\in
\dom\varphi^*\}$. Another consequence of cash-invariance
(\ref{eq:CashAdd}) is that it implies $\varphi$ is finite on
$L^\infty$, since then $-\|X\|_\infty=\varphi(-\|X\|_\infty)\leq
\varphi(X)\leq \varphi(\|X\|_\infty)=\|X\|_\infty$ for all $X\in
L^\infty$ by the monotonicity and (\ref{eq:CashAdd}).  Thus all of our
main results apply to any Lebesgue convex risk functions on
$L^\infty$. Note also that any Lebesgue extension of a convex risk
function $\varphi_0$ on $L^\infty$ retains the cash-invariance
(\ref{eq:CashAdd}) since if $(\varphi,\Xs)$ is a Lebesgue extension of
such $\varphi_0$,
\begin{align*}
  \varphi(X+c)&=\lim_n\varphi(X\mathds1_{\{|X|\leq n\}}+c)
  =\lim_n\varphi_0(X\mathds1_{\{|X|\leq n\}}+c)\\
  & =\lim_n\varphi_0(X\mathds1_{\{|X|\leq n\}})+c
  =\lim_n\varphi(X\mathds1_{\{|X|\leq n\}})+c
\end{align*}
Consequently we have the following as a paraphrasing of
Theorem~\ref{thm:MainLebExt1}:
\begin{corollary}
  \label{cor:ConvRiskFunc}
  Let $\varphi_0$ be a convex risk function on $L^\infty$ with the
  Lebesgue property and $\varphi_0(0)=0$, $\varphi^*_0$ its conjugate,
  $\QC_0:=\QC_{\varphi_0}$ and $\DC_0:=\{X\in L^0:\, X^-\in
  L^1(Q),\,\forall Q\in \QC_0\}$. Then we have:
  \begin{enumerate}
  \item The following are well-defined
    \begin{align}
      \label{eq:ConvRiskFuncHat}
      \hat\varphi(X)&=\sup_{Q\in \QC_0}(\EB_Q[X]-\varphi_0^*(Q)),\,X\in \DC_0\\
      M^{\hat\varphi}_u&=\{X\in L^0:\,
      \lim_N\hat\varphi(\alpha|X|\mathds1_{\{|X|>N\}})=0,\,\forall
      \alpha>0\} \subset \DC_0\cap (-\DC_0).
    \end{align}
  \item $\hat\varphi$ is a finite convex risk function on
    $M^{\hat\varphi}_u$ with the Lebesgue property and
    $\hat\varphi|_{L^\infty}=\varphi_0$, and for any other pair
    $(\varphi,\Xs)$ of a solid space $\Xs\subset L^0$ and a convex
    risk function on $\Xs$ with the Lebesgue property and
    $\varphi|_{L^\infty}=\varphi_0$, we have $\Xs\subset
    M^{\hat\varphi}_u$ and $\varphi=\hat\varphi|_{\Xs}$.
  \end{enumerate}
\end{corollary}

Note that the assumption $\varphi_0(0)=0$ is just for notational
simplicity; without this assumption,
$(\hat\varphi,M^{\hat\varphi-\varphi_0(0)}_u)$ is the maximum Lebesgue
extension of $\varphi_0$.

Here we examine some typical risk functions deriving the explicit
forms of the space $M^{\hat\varphi}_u$.  We begin with a simple
remark. Though we defined $\hat\varphi$ using the dual representation
of $\varphi_0$ on $L^\infty$, it may be more convenient to use other
more explicit formula for $\varphi_0$ if available. By
Lemma~\ref{lem:HatPhiMon1}, we know that $\hat\varphi$ is continuous
from below on $\DC_0\supset L^0_+$. In particular,
\begin{equation}
  \label{eq:ExampleFormulaExplicitGeneral}
  \hat\varphi(X)=\lim_n\varphi_0(X \wedge n),\quad \forall X\in L^0_+.
\end{equation}
Note that this formula may not be true for $X\in \DC_0\setminus
L^0_+$, but we need only consider $|X|$ with $X\in L^0$ to derive the
spaces $M^{\hat\varphi}_u$ and $M^{\hat\varphi}$.

\begin{example}[Entropic Risk Function]
  \label{ex:entropy1}
  Let
  \begin{equation}
    \label{eq:Entropic1}
    \varphi_{\mathrm{ent}}(X):=\log \EB[\exp(X)],\quad X\in L^\infty.    
  \end{equation}
  This is called the \emph{entropic risk function}. It is
  straightforward from the dominated convergence theorem that
  $\varphi_{\mathrm{ent}}$ has the Lebesgue property on
  $L^\infty$. Its conjugate $\varphi_{\mathrm{ent}}^*$ is given as
  $\varphi^*_{\mathrm{ent}}(Q)=\HC(Q|\PB):=\EB[(dQ/d\PB)\log(dQ/d\PB)]$,
  the \emph{relative entropy} (thus \emph{entropic}), hence we have
  \begin{align*}
    \hat \varphi_{\mathrm{ent}}(X)=\sup_{Q\ll\PB,
      \HC(Q|\PB)<\infty}(\EB_Q[X]-\HC(Q|\PB)),%=\log \EB[\exp(X)],
  \end{align*}
  and the identity $\hat\varphi_{\mathrm{ent}}(X)=\log\EB[\exp(X)]$
  remains true for all $X\in L^0_+$.  In particular,
  $M^{\hat\varphi_{\mathrm{ent}}}=M^{\Phi_{\exp}} \subsetneq
  L^{\Phi_{\exp}}=L^{\hat\varphi_{\mathrm{ent}}}$ if
  $(\Omega,\FC,\PB)$ is atomless, where $\Phi_{{\exp}}(x)=e^x-1$
  ($x\geq 0$) and $M^{\Phi_{\exp}}$ (resp. $L^{\Phi_{\exp}}$) is the
  associated Orlicz heart (resp. space). Further, we see that
  $M^{\hat\varphi_{\mathrm{ent}}}_u
  =M^{\hat\varphi_{\mathrm{ent}}}(=M^{\Phi_{\exp}})$, since if $X\in
  M^{\Phi_{\exp}}$,
  $\exp(\varphi_{\mathrm{ent}}(\lambda|X|\mathds1_{\{|X|>N\}}))=\EB[\exp(\lambda
  |X|\mathds1_{\{|X|>N\}})]=\EB[\exp(\lambda|X|)\mathds1_{\{|X|>N\}}]+\PB(|X|\leq
  N)\rightarrow 1$ by the dominated convergence for every $\lambda>0$.

\end{example}

\subsection{Utility Based Shortfall Risk}
\label{sec:Shortfall}

Let $l:\RB\rightarrow\RB$ be a (finite) increasing convex function
with $l(0)>\inf_xl(x)$ (thus not identically constant). Then its
conjugate $l^*(y)=\sup_x(xy-l(x))$ is a convex function with
\begin{equation}
  \label{eq:LConj1}
  \mathrm{dom}l^*\subset \RB_+\quad\text{and}\quad \lim_{y\uparrow \infty}\frac{l^*(y)}y=+\infty.
\end{equation}
The second one shows also that for any $c>0$, there exist
$\underline{\Lambda}(c), \overline\Lambda(c)\in(0,\infty)$ such that
\begin{equation}
  \label{eq:LambdaBounds}
  \frac{l(0)+l^*(y)}{y}\leq c+1\,\Rightarrow\, y\in [\underline\Lambda(c),\overline\Lambda(c)]
\end{equation}
Indeed, if $I_c:=\{y>0:\, (l(0)+l^*(y))/y\leq c+1\}$ is empty, put
$\underline\Lambda(c)=\overline\Lambda(c)=1$. Otherwise,
$\overline\Lambda(c):=\sup I_c$ is finite by (\ref{eq:LConj1}), and
picking $x_0<0$ with $l(x_0)<l(0)$ (by assumption),
\begin{align*}
  \frac{l(0)+l^*(y)}y=\sup_x\left(x+\frac{l(0)-l(x)}y\right)\geq
  x_0+\frac{l(0)-l(x_0)}y,
\end{align*}
hence $\underline\Lambda(c)=\frac{l(0)-l(x_0)}{c+1-x_0}>0$ does the
job.

Now we define the associated \emph{shortfall risk function} by
\begin{equation}
  \label{eq:ShortfallRisk1}
  \varphi_l(X):=\inf\{x\in\RB:\, \EB[l(X-x)]\leq l(0)\},\quad \forall X\in L^\infty.
\end{equation}
This is a convex risk function with the Lebesgue property
(\ref{eq:LebLinfty1}) and its conjugate is
\begin{equation}
  \label{eq:PenaltyShortfall}
  \varphi_l^*(Q):=\varphi^*_l(dQ/d\PB)
  =\inf_{\lambda>0}\frac1\lambda \left(l(0)+\EB\left[l^*\left(\lambda\frac{dQ}{d\PB}\right)\right]\right).
\end{equation}
% where $l^*(y)=\sup_x(xy-l(x))$.
(See \citep[Ch.4]{MR2779313}).  Also,
(\ref{eq:ExampleFormulaExplicitGeneral}) implies that
\begin{align*}
  \hat\varphi_l(|X|)&=\sup_n\inf\{x:\, \EB[l(|X|\wedge n-x)]\leq
  l(0)\} \leq \inf\{x:\, \EB[l(|X|-x)]\leq l(0)\},
\end{align*}
while if $\hat\varphi_l(|X|)<\infty$, then
$\EB[l(|X|-\hat\varphi_l(|X|))] \leq \lim_n\EB[l(|X|\wedge
n-\hat\varphi_l(|X|))] \leq \limsup_n\EB[l(|X|\wedge
n-\varphi_l(|X|\wedge n))] \leq l(0)$ by monotone convergence and
$\varphi_l(|X|\wedge n)\leq \hat\varphi_l(|X|)$, thus
\begin{align}\label{eq:Shortfall2} \hat\varphi_l(|X|)=\inf\{x:\,
  \EB[l(|X|-x)]\leq l(0)\}, X\in L^0.
\end{align}

In this case, two spaces $M^{\hat\varphi_l}_u$ and $M^{\hat\varphi_l}$
coincide and equal to the Orlicz heart associated to the Young
function $\Phi_l(x):=l(|x|)-l(0)$, i.e.,
\begin{proposition}
  \label{prop:ExShortfall1}
  $M^{\hat\varphi_l}_u=M^{\hat\varphi_l}=M^{\Phi_l}$.
\end{proposition}
\begin{proof}
  To see $M^{\Phi_l}\subset M^{\hat\varphi_l}_u$, it suffices that
  $\{XdQ/d\PB:\, \varphi^*_l(Q)\leq c\}$ is uniformly integrable for
  any $c>0$ and $X\in M^{\Phi_l}$ by Theorem~\ref{thm:MuCharact} . So
  let us fix $c>0$ and $X\in M^{\Phi_l}$. Observe that if
  $\varphi^*_l(Q)\leq c$, then there exists a $\lambda_Q>0$ such that
  \begin{equation}
    \label{eq:ProofShortfall1}
    c+1 \geq \frac{1}{\lambda_Q} 
    \left(l(0) +\EB\left[l^*\left(\lambda_Q\frac{dQ}{d\PB}\right)\right]\right) 
    \geq \frac{l(0)+l^*(\lambda_Q)}{\lambda_Q}
  \end{equation}
  by (\ref{eq:PenaltyShortfall}) and Jensen's inequality, and then
  $\lambda_Q\in [\underline\Lambda(c),\overline\Lambda(c)]$ by
  (\ref{eq:LambdaBounds}). Since $l(\alpha|X|\mathds1_A)
  =\Phi_l(\alpha|X|)\mathds1_A+l(0)$, Young's inequality shows for any
  $A\in\FC$, $\alpha>0$ and $Q$ with $\varphi^*_l(Q)\leq c$,
  \begin{align*}
    \EB_Q[|X|\mathds1_A] &\leq \frac1{\alpha \lambda_Q}\left(
      \EB[\Phi_l(\alpha|X|)\mathds1_A]+\left(l(0)
        +\EB\left[l^*\left(\lambda_Q\frac{dQ}{d\PB}\right)\right]\right)\right)\\
    &\leq\frac1{\alpha\lambda_Q}\EB[\Phi_l(\alpha|X|)\mathds1_A]+\frac{c+1}\alpha
    \leq
    \frac1{\alpha\underline\Lambda(c)}\EB[\Phi_l(\alpha|X|)\mathds1_A]+\frac{c+1}\alpha.
  \end{align*}
  Since $X\in M^{\Phi_l}$, the desired uniform integrability follows
  from a diagonal argument.

  On the other hand, note that $l(\alpha|X|/2)\leq\frac{1}{2}l(\alpha
  |X|-x)+\frac12 l(x)$ by convexity, hence $M^{\hat\varphi_l}\subset
  M^{\Phi_l}$ follows from (\ref{eq:Shortfall2}), and we deduce that
  the three spaces agree.
\end{proof}
% $\frac1{\lambda_Q}\left(l(0)+\EB\left[l^*(\lambda_Q
 %      dQ/d\PB)\right]\right)\leq c+1$, and any such $\lambda_Q$ is
 %  bounded below by a constant depending only on $c$ and $l$. Indeed,

 %  by Jensen's inequality. Taking $x_0<0$ so that $l(x_0)<l(0)$ (such
 %  exists since $l(0)>\inf_xl(x)$),
 %  $(l(0)+l^*(\lambda_Q))/\lambda_Q=\sup_x(x+(l(0)-l(x))/\lambda_Q)\geq
 %  x_0+(l(0)-l(x_0))/\lambda_Q$, hence (\ref{eq:ProofShortfall1})
 %  implies $\lambda \geq \frac{l(0)-l(x_0)}{c+1-x_0}=:\Lambda(c)$. On
 %  the other hand, noting that 

 %  from which we have
 % \begin{align*}
 %   \EB_Q[|X|\mathds1_A]&\leq\frac{1}{\alpha\lambda_Q}\EB[\Phi_l(\alpha|X|)\mathds1_A]
 %   +\frac{c+1}\alpha \leq
 %   \frac{1}{\alpha}\frac{\EB[\Phi_l(\alpha|X|)\mathds1_A]}{
 %     \Lambda(c)}+\frac{c+1}{\alpha}
 % \end{align*}
 % for any $Q$ with $\varphi^*_l(Q)\leq c$. 

 % $M^{\hat\varphi_l}\subset M^{\Phi_l}$ (hence three spaces agree)
 % follows from (\ref{eq:Shortfall2}). Indeed, we have the implications:
 % $\hat\varphi_l(\alpha |X|)<\infty$ $\Rightarrow$ $\exists x\in\RB$
 % s.t. $l(\alpha |X|-x)\in L^1$ $\Rightarrow$
 % $l(\alpha|X|/2)\leq\frac{1}{2}l(\alpha |X|-x)+\frac12 l(x)\in
 % L^1$. We deduce that $M^{\hat\varphi}\subset M^{\Phi_l}$.
%\end{proof}

\begin{remark}
  In definition (\ref{eq:ShortfallRisk1}), we have chosen $l(0)$ for
  the acceptance level so that $\varphi_l(0)=0$. If $\varphi_l$ is
  defined with other acceptance level $\delta$ instead of $l(0)$, we
  can normalize it by adding the constant
  $a^l(\delta):=\sup\{x:l(x)\leq \delta\}$ or equivalently replacing
  the function $l$ by $x\mapsto l(x+a^l(\delta))$. The case
  $l(0)=\inf_xl(x)$ corresponds to the worst case risk function
  $\varphi^{\text{worst}}(X)=\esssup X$. Also, if $l(x)=e^x$, then
  $\varphi_l=\varphi_{\mathrm{ent}}$.
\end{remark}

\subsection{Robust Shortfall Risk}
\label{sec:ExRobShortfall}

Let $l$ be as above and fix a set $\PC$ of probabilities $P\ll \PB$
such that
\begin{equation}
  \label{eq:PCompact}
  \PC\text{ is convex and weakly compact in }L^1.
\end{equation}
Then we consider a \emph{robust shortfall risk function}
\begin{equation}
  \label{eq:RobustSF1}
  \varphi_{l,\PC}(X):=\inf\{x\in\RB:\, \sup_{P\in\PC}\EB_P[l(X-x)]\leq l(0)\}, \quad X\in L^\infty.
\end{equation}
The function $\varphi_{l,\PC}$ on $L^\infty$ is a convex risk function
whose conjugate is given by
\begin{equation}
  \label{eq:RobShortfall1}
  \varphi^*_{l,\PC}(Q) :=\inf_{\lambda>0} 
  \frac1\lambda\left(l(0)+\inf_{P\in\PC}\EB_P\left[l^*\left(\lambda \frac{dQ}{dP}\right)\right]\right)
\end{equation}
with the convention $l^*(\infty):=\infty$ and $ \frac{dQ}{dP}
:=\frac{dQ/d\PB}{dP/d\PB}\mathds1_{\{dP/d\PB>0\}}
+\infty\cdot\mathds1_{\{dQ/d\PB>0,dP/d\PB=0\}}$ (see
\citep[Corollary~4.119]{MR2779313}). Slightly modifying the argument
for (\ref{eq:Shortfall2}), we still have
\begin{equation}
  \label{eq:RobustSF2}
  \hat\varphi_{l,\PC}(|X|)=\inf\{x\in\RB:\, \sup_{P\in\PC}\EB_P[l(|X|-x)]\leq l(0)\}, \quad X\in L^0.
\end{equation}

% Under (\ref{eq:PCompact}), we
% have further that $\varphi_{l,\PC}$ has the Lebesgue property on
% $L^\infty$. This follows from a robust version of de la Vallée-Poussin
% theorem due to \citep[Lemma~2.12]{MR2247836} (this result is
% stated there for sets of \emph{probability measures}, but their proof
% does not use the latter fact, and the exactly same proof applies to
% sets of \emph{positive finite measures}).  Also, slightly 

We introduce a couple of ``robust analogues'' of $M^{\Phi_l}$:
\begin{align*}
  M^{\Phi_l}(\PC) &:=\{X\in L^0:\,
  \sup_{P\in\PC}\EB_P[\Phi_l(\lambda|X|)]<\infty,\,\forall
  \lambda>0\}\\
  M^{\Phi_l}_u(\PC)&:=\{X\in L^0:\,
  \lim_{N\rightarrow\infty}\sup_{P\in\PC}\EB_P[\Phi_l(\lambda|X|)\mathds1_{\{|X|>N\}}]=0,\,\forall
  \lambda>0\}.
\end{align*}
When $\PC=\{\PB\}$, the two spaces coincide with $M^{\Phi_l}$. Now we
have:
\begin{proposition}
  \label{prop:RobShortfall}
  Assume (\ref{eq:PCompact}). Then $\varphi_{l,\PC}$ is Lebesgue on
  $L^\infty$ and
  \begin{align*}
    M^{\Phi_l}_u(\PC)= M^{\hat\varphi_{l,\PC}}_u\subset
    M^{\hat\varphi_{l,\PC}}\subset M^{\Phi_l}(\PC).
  \end{align*}
\end{proposition}
\begin{proof}
  With a similar reasoning as Proposition~\ref{prop:ExShortfall1}, if
  $Q\in \QC_c:=\{Q\ll\PB:\varphi^*_{l,\PC}(Q)\leq c\}$, there exist
  $\lambda_Q\in [\underline\Lambda(c), \overline\Lambda(c)]$ and
  $P_Q\in\PC$ such that
  \begin{align*}
    \frac1{\lambda_Q}\left(l(0)+\inf_{P\in\PC}\EB_P\left[l^*\left(\lambda_Q\frac{dQ}{dP}\right)\right]\right)
    & \leq
    \frac1{\lambda_Q}\left(l(0)+\EB_{P_Q}\left[l^*\left(\lambda_Q\frac{dQ}{dP_Q}\right)\right]\right)\leq
    c+1.
  \end{align*}
  In particular, $\inf_{P\in\PC}\EB_P[l^*(\lambda_Q dQ/dP)]\leq
  \overline\Lambda(c)(c+1)-l(0)$ whenever $Q\in\QC_c$. In view of
  (\ref{eq:LConj1}), this shows that $\{\lambda_QdQ/d\PB:Q\in\QC_c\}$
  is uniformly integrable thanks to the robust version of de la
  Vallée-Poussin theorem \citep[][Lemma~2.12]{MR2247836} (which is
  stated there for sets of probabilities, but the exactly same proof
  works for sets of positive finite measures), hence so is $\QC_c$
  since $\lambda_Q\geq \underline\Lambda(c)$ for each
  $Q\in\QC_c$. Consequently, $\varphi_{l,\PC}$ is Lebesgue on
  $L^\infty$ by the JST theorem (Theorem~\ref{thm:JSTLinfty}).

  From the same inequality, we see also that
  \begin{align*}
    \EB_Q[|X|\mathds1_A]&\leq \frac1{\alpha\lambda_Q}
    \left(\EB_{P_Q}[\Phi(\alpha|X|)\mathds1_A]+l(0)
      +\EB_{P_Q}\left[l^*\left(\lambda_Q\frac{dQ}{dP_Q}\right)\right]\right)\\
    &\leq\frac1{\alpha\underline\Lambda(c)}\sup_{P\in\PC}\EB_P[\Phi(\alpha|X|)\mathds1_A]
    +\frac{c+1}\alpha
  \end{align*}
  for any $\alpha>0$, $A\in\FC$ and $Q\in\QC_c$. Hence if $X\in
  M^{\Phi_l}_u(\PC)$, a diagonal argument shows that $\{XdQ/d\PB:
  Q\in\QC_c\}$ is uniformly integrable, hence
  $M^{\Phi_l}_u(\PC)\subset M^{\hat\varphi_{l,\PC}}_u$ by
  Theorem~\ref{thm:MuCharact}.

  % To show that $M^{\Phi_l}_u(\PC)\subset M^{\hat\varphi_{l,\PC}}_u$,
  % fix $X\in M^{\Phi_l}_u(\PC)$. From the above paragraph, we see that

  % it suffices that $\{XdQ/d\PB:\, \varphi^*_{l,\PC}(Q)\leq c\}$ is
  % uniformly integrable for all $X\in M^{\Phi_l}_u(\PC)$ and
  % $c>0$. With a similar reasoning and notation as
  % Proposition~\ref{prop:ExShortfall1}, we see that
  % $\varphi^*_{l,\PC}(Q)\leq c$ implies the existence of $\lambda_Q\geq
  % \Lambda(c)=(l(0)-l(x_0))/(c+1-x_0)$ and $P_Q\in \PC$ such that
  % $\frac{1}{\lambda_Q}\left(l(0) +\EB_{P_Q}\left[l^*\left(\lambda_Q
  %       \frac{dQ}{dP_Q}\right)\right]\right)\leq c+1$.  By Young's
  % inequality and
  % $l(\alpha|X|\mathds1_A)=\Phi_l(\alpha|X|)\mathds1_A+l(0)$, we see
  % that
  % \begin{align*}
  %   \EB_Q[\lambda_Q\alpha|X|\mathds1_A]&\leq
  %   \EB_{P_Q}[\Phi_l(\alpha|X|)\mathds1_A]
  %   +\left(l(0)+\EB_{P_Q}\left[l^*\left(\lambda_Q\frac{dQ}{dP_Q}\right)\right]\right)\\
  %   &\leq\sup_{P\in\PC}\EB_{P}[\Phi_l(\alpha|X|)\mathds1_A]+\lambda_Q
  %   (c+1)
  % \end{align*}
  % for all $Q$ with $\varphi^*_{l,\PC}(Q)\leq c$, $A\in\FC$ and
  % $\alpha>0$. Hence
  % \begin{align*}
  %   & \sup\{\EB_Q[|X|\mathds1_A]:\varphi^*_{l,\PC}(Q)\leq c\} \leq
  %   \frac{1}{\alpha
  %     \Lambda(c)}\sup_{P\in\PC}E[\Phi_l(\alpha|X|)\mathds1_A]
  %   +\frac{c+1}{\alpha},
  % \end{align*}
  % from which the uniform integrability follows by a diagonal
  % technique. 

  To see $M^{\Phi_l}_u(\PC)\supset M^{\hat\varphi_{l,\PC}}_u$, let
  $X\in M^{\hat\varphi_{l,\PC}}_u$ and $\alpha>0$. By the definition
  of $M^{\hat\varphi_{l,\PC}}_u$, there is a sequence $(N_n)_n\subset
  \NB$ such that $\hat\varphi_{l,\PC}(n
  \alpha|X|\mathds1_{\{|X|>N_n\}})< 2^{-n}$. Then by
  (\ref{eq:RobustSF2}),
  \begin{align*}
    \sup_{P\in\PC}\EB_P[l(n\alpha|X|\mathds1_{\{|X|>N_n\}}-2^{-n})]\leq
    l(0).
  \end{align*}
  Noting that
  $\Phi_l(\alpha|X|\mathds1_{A_n})=l(\alpha|X|\mathds1_{A_n})-l(0)\leq
  n^{-1}l(n\alpha|X|\mathds1_{A_n}-2^{-n})+\frac{n-1}nl(\frac{2^{-n}}{n-1})-l(0)$
  with $A_n:=\{|X|>N_n\}$ by the convexity, we have
  \begin{align*}
    \sup_{P\in\PC}\EB_P[\Phi_l(\alpha|X|)\mathds1_{A_{n}}]
    &\leq\frac{1}{n}\sup_{P\in\PC}
    \EB_P[l(n\alpha|X|\mathds1_{A_{n}}-2^{-n})]+\frac{n-1}{n}l\left(\frac{2^{-n}}{n-1}\right)-l(0) \\
    &\leq
    \frac{l(0)}{n}+\frac{n-1}{n}l\left(\frac{2^{-n}}{n-1}\right)-l(0)
    \rightarrow 0+l(0)-l(0)=0.
  \end{align*}
  Since $\alpha>0$ is arbitrary, we have $X\in M^{\Phi_l}_u(\PC)$.

  Finally, we show $M^{\hat\varphi_{l,\PC}}\subset M^{\Phi_l}(\PC)$.
  If $X\in M^{\hat\varphi_{l,\PC}}$, then for every $\alpha>0$,
  \begin{align*}
    \sup_{P\in\PC}\EB_P[\Phi_l(\alpha|X|)]
    &=\sup_{P\in\PC}\EB_P[l(\alpha|X|)]-l(0)\\
    & \leq
    \frac{1}{2}\sup_{P\in\PC}\EB_P[l(2\alpha|X|-x)]+\frac{1}{2}l(x)-l(0)<\infty.
  \end{align*}
  for $x>\hat\varphi_{l,\PC}(\alpha|X|)$ by (\ref{eq:RobustSF2}). Thus
  $M^{\hat\varphi_{l,\PC}}\subset M^{\Phi_l}(\PC)$.
\end{proof}

\begin{example}[Robust Entropic Risk Functions]
  Let $l(x)=e^x$. Then $\varphi_{l,\PC}$ is the entropic one, and
  the associated Young function is $\Phi_{\exp}(x)=e^{x}-1$. In this
  case, we have $M^{\Phi_{\exp}}_u(\PC)=M^{\Phi_{\exp}}(\PC)$,
  thus
  $M^{\hat\varphi_{l,\PC}}_u=M^{\hat\varphi_{l,\PC}}$. Indeed, by
  Hölder's inequality,
  \begin{align*}
    \sup_{P\in\PC}\EB_P[e^{\alpha|X|}\mathds1_{\{|X|>N\}}] &\leq
    \sup_{P\in\PC}\left(\EB_P[e^{2\alpha|X|}]^{1/2}P(|X|>N)^{1/2}\right)\\
    &\leq
    \sup_{P\in\PC}\EB_P[e^{2\alpha|X|}]^{1/2}\sup_{P\in\PC}P(|X|>N)^{1/2}.
  \end{align*}
  This and the uniform integrability of $\PC$ show that
  $\lim_N\sup_{P\in\PC}\EB_P[e^{\alpha|X|}\mathds1_{\{|X|>N\}}]=0$ for
  every $\alpha>0$ as soon as $X\in M^{\Phi_{\exp}}(\PC)$, hence
  $M^{\Phi_{\exp}}_u(\PC)=M^{\Phi_{\exp}}(\PC)$.
\end{example}

\subsection{Law-Invariant Case}
\label{sec:LawInv}

Recall that a convex risk function $\varphi_0$ on $L^\infty$ is called
\emph{law-invariant} if $\varphi_0(X)=\varphi_0(Y)$ whenever $X$ and
$Y$ have the same distribution. Any law-invariant convex risk function
on $L^\infty$ has the following \emph{Kusuoka representation}
(\citep{MR1886557}, \citep{MR2766847}):
\begin{align}\label{eq:Kusuoka1}
  \varphi_0(X)=\sup_{\mu\in\MC_1((0,1])}\left(\int_{(0,1]}v_\lambda
    (X)\mu(d\lambda)-\beta(\mu)\right)
\end{align}
where $v_\lambda(X):=\frac{1}{\lambda}\int_0^\lambda q_X(1-t)dt$, the
\emph{average value at risk at level $\lambda$} (up to change of
sign), $q_X(t):=\inf\{x: \PB(X\leq x)>t\}$, $\MC_1((0,1])$ is the set
of all Borel probability measures on $(0,1]$ and $\beta$ is a lower
semi-continuous penalty function.  Then $\varphi_0$ has the Lebesgue
property on $L^\infty$ if and only if all the level sets
$\{\mu:\,\beta(\mu)\leq c\}$ are relatively weak* compact in
$\MC_1((0,1])$ or equivalently tight
(\citep[Ch.~5]{delbaen12:_monet_utilit_funct} or
\citep{jouini06:_law_fatou}). In particular, for any relatively weak*
compact convex set $\MC\subset \MC_1((0,1])$,
\begin{align*}
  \varphi_\MC(X):=\sup_{\mu\in\MC}\int_{(0,1]}v_\lambda(X)\mu(d\lambda)
\end{align*}
is a law-invariant coherent risk function on $L^\infty$ satisfying the
Lebesgue property.

\begin{example}[AV@R]\label{ex:AVAR}
  For every $\lambda \in (0,1]$, $v_\lambda$ admits the
  representation:
  \begin{equation}
    \label{eq:AVaR2}
    v_\lambda(X)=\sup\{\EB_Q[X]:\,Q\in\PC,\, dQ/d\PB\leq 1/\lambda\},
  \end{equation}
  for all $X\in L^\infty$, and since $\hat
  v_\lambda(|X|)=\sup_{n}v_\lambda(|X|\wedge n)$,
  \begin{align*}
    \|X\|_{L^1}\leq \hat v_\lambda(|X|)=\|X\|_{\hat v_\lambda}\leq
    \frac1\lambda \|X\|_{L^1},\quad X\geq 0.
  \end{align*}
  Hence we have $M^{\hat v_\lambda}_u=M^{\hat v_\lambda}=L^1$ for
  every $\lambda\in (0,1]$, and the representation (\ref{eq:AVaR2})
  extends to $L^1$. In particular, $\hat v_\lambda$ has the Lebesgue
  property on $L^1$.
    
\end{example}

\begin{example}[Concave Distortions]
  Let $\mu\in \MC_1((0,1])$ and define
  \begin{align*}
    \varphi_\mu(X):=\int_{(0,1]}v_t(X)\mu(dt).
  \end{align*}
  This type of risk functions are called concave distortion, and it is
  known that if the probability space $(\Omega,\FC,\PB)$ is atomless,
  every law-invariant \emph{comonotonic} risk function is written in
  this form (see \citep[][Theorem~4.93]{MR2779313}). For
  $\varphi_\mu$, two spaces $M^{\hat\varphi_\mu}_u$ and
  $M^{\hat\varphi_\mu}$ coincide. Indeed, if
  $\hat\varphi_\mu(|X|)<\infty$ ($\Leftrightarrow$ $\hat v_\cdot
  (|X|)\in L^1((0,1],\mu)$), then from Example~\ref{ex:AVAR}, we see
  that $v_t(|X|\mathds1_{\{|X|>N\}})\leq v_t(|X|)$ and
  $\lim_Nv_t(|X|\mathds1_{\{|X|>N\}})=0$ for ($\mu$-a.e., hence) all
  $t\in(0,1]$. Thus the dominated convergence theorem implies that
  % implies that , hence $\hat v_t(|X|)<\infty$ for $\mu$-a.e. $t\in
  % (0,1]$.  Since $\hat v_t(|X|\mathds1_{\{|X|>N\}})\downarrow 0$ as
  % $N\rightarrow\infty$ and $\hat v_\cdot
  % (|X|\mathds1_{\{|X|>N\}})\leq v_\cdot(|X|)\in L^1((0,1],\mu)$ for
  % $\mu$-a.e.  $t\in (0,1]$, the dominated convergence theorem shows
  % that
  \begin{align*}
    \lim_N\int_{(0,1]}\hat v_t(|X|\mathds1_{\{|X|>N\}})\mu(dt)=
    \int_{(0,1]}\lim_N\hat v_t(|X|\mathds1_{\{|X|>N\}})\mu(dt)=0.
  \end{align*}
  Repeating the same argument for $\alpha|X|$ ($\alpha>0$) instead of
  $X$, we have $M^{\hat\varphi_\mu}_u=M^{\hat\varphi_\mu}$.
\end{example}

Recall that any \emph{finite-valued} convex risk function on a solid
and rearrangement-invariant space strictly bigger than $L^\infty$ has
the Lebesgue property \emph{restricted to} $L^\infty$
(\citep[][Theorem~3]{MR2509290} or see the comment after
Theorem~\ref{thm:JSTLinfty}). The next example concerns how is the
Lebesgue property on the whole space. In our context, both
$M^{\hat\varphi}_u$ and $M^{\hat\varphi}$ are (solid and)
rearrangement-invariant if the $\varphi_0$ is law-invariant, and
$M^{\hat\varphi}$ is the maximum solid vector space on which
$\hat\varphi$ is finite-valued. Then the question is translated as:
does it hold $M^{\hat\varphi}_u=M^{\hat\varphi}$ \emph{as soon as
  $\varphi_0$ is law-invariant}?  The answer is generally no.

\begin{example}[A law-invariant risk function with
  $M^{\hat\varphi}_u\subsetneq
  M^{\hat\varphi}$]\label{ex:LawInvariantNonUI}
  Let $(\Omega,\FC,\PB)$ be atomless and for each $n$, we define a
  Borel probability measure on $(0,1]$ by
  \begin{align}\label{eq:ExLawinvCountFam1}
    \mu_n(dt)
    :=\left(1-\frac{1}{n}\right)\frac{e}{e-1}\mathds1_{(e^{-1},1]}(t)dt
    +\frac{1}{n}\frac{e^n}{e-1}\mathds1_{(e^{-n},e^{-n+1}]}(t)dt.
  \end{align}
  Then $(\mu_n)_n$ (and hence $\overline{\mathrm{conv}}(\mu_n;
  n\in\NB)$) is uniformly integrable in $L^1((0,1],dt)$
  ($\Leftrightarrow$ weak* compact in $\MC_1((0,1])$). Hence the
  law-invariant coherent risk function
  \begin{align*}
    \varphi_0(X):=\sup_n\int_{(0,1]}v_t(X)\mu_n(dt)\quad \left(
    \Rightarrow\, \hat\varphi(|X|)=\sup_n\int_{(0,1]}\hat v_\lambda
    (|X|)\mu_n(d\lambda)\right)
  \end{align*}
  has the Lebesgue property on $L^\infty$. In this case,
  $M^{\hat\varphi}_u\subsetneq M^{\hat\varphi}$. Indeed, let $X$ be an
  exponential random variable with parameter 1, i.e.,
  $F_X(x):=\PB(X\leq x)=1-e^{-x}$ $\Leftrightarrow$
  $q_X(t)=-\log(1-t)$. Then
  \begin{align*}
    \hat v_\lambda(X)&=\frac1\lambda\int_0^\lambda (-\log
    t)dt=1-\log\lambda.
  \end{align*}
  For each $n$, $\int_{(0,1]}\hat v_t(X)\mu_n(dt)
  =4-\frac{e}{e-1}-\frac1n$, so $\hat\varphi(X)= \sup_n\int_{(0,1]}\hat
  v_t(X)\mu_n(dt)=4-\frac{e}{e-1}<\infty$. This shows that $X\in
  M^{\hat\varphi}$. We next compute $\lim_N\varphi(X\mathds1_{\{X>N\}})$.
  Since $q_{X\mathds1_{\{X>N\}}}(t)=q_X\mathds1_{\{q_X(t)>N\}}$ and
  $q_X(1-t)>N$ $\Leftrightarrow$ $t<1-F_X(N)=e^{-N}$,
  \begin{align*}
    \hat v_\lambda(X\mathds1_{\{X>N\}}) &=\frac1\lambda
    \int_0^\lambda q_X(1-t)\mathds1_{\{q_X(1-t)>N\}}dt\\
    & =\{\lambda\wedge e^{-N}-(\lambda\wedge e^{-N})\log
    (\lambda\wedge e^{-N}))\}/\lambda.
  \end{align*}
  Thus for $n>N+1$,
  \begin{align*}
    &       \int_{(0,1]}\hat v_t(X\mathds1_{\{X>N\}})\mu_n(dt)\\
    &\quad =\left(1-\frac1n\right)\frac{e}{e-1}\left(e^{-N}-e^{-N}\log
      e^{-N}\right)+\frac1n \left(2+n-\frac{e}{e-1}\right)\\
    &\quad =1+\frac{e}{e-1}\left(e^{-N}-e^{-N}\log e^{-N}\right)
    +\frac1n\left\{2-\frac{e}{e-1}\left(1+e^{-N}-e^{-N}\log
        e^{-N}\right)\right\}
  \end{align*}
  Hence $\hat\varphi(X\mathds1_{\{X>N\}}) =\sup_n\int_{(0,1]}\hat
  v_t(X\mathds1_{\{X>N\}})\mu_n(dt) =
  1+\frac{e}{e-1}\left(e^{-N}-e^{-N}\log e^{-N}\right)$. Consequently,
  $\lim_{N\rightarrow \infty}\varphi(X\mathds1_{\{X>N\}})\geq
  1+\lim_N\frac{e}{e-1}\left(e^{-N}-e^{-N}\log e^{-N}\right)=1$.  Thus
  $X\not\in M^{\hat\varphi}_u$.
\end{example}

\section*{Acknowledgements}

The author warmly thanks Takuji Arai for numerous discussions and
comments. He also thank an anonymous referee for carefully reading the
manuscript. The financial support of the Center for Advanced Research
in Finance (CARF) at the Graduate School of Economics of the
University of Tokyo is gratefully acknowledged.

\appendix
% \section{A Minimax Theorem}
% \label{sec:Minimax}
%\let\sectionname{}
%\left\thesection{\@Alph\c@section}
\addcontentsline{toc}{section}{Appendix}
\section*{Appendix}
\label{sec:Omitted}

We have used the following version of minimax theorem which should be
known as it is an immediate corollary to \citep[][Theorems~1 and
2]{MR764631}. But we could not find an appropriate reference, so we
include here a simple proof.
\begin{theorem}
  \label{thm:AppMinimax}
  Let $C$ be a convex subset of a Hausdorff topological vector space,
  and $D$ an arbitrary convex set. Suppose we are given a function
  $f:C\times D\rightarrow\RB$ such that
  \begin{enumerate}
  \item for any $y\in D$, $x\mapsto f(x,y)$ is convex and
    % level-compact, i.e.,
    $\{x\in C:\, f(x,y)\leq c\}$
    % $\mathrm{lev}_{\leq\alpha}f(\cdot,y):=\{x\in C:\, f(x,y)\leq
    % \alpha\}$
    is compact for each $c\in\RB$;
  \item for any $x\in C$, $y\mapsto f(x,y)$ is concave on $D$.
  \end{enumerate}
  Then we have
  \begin{equation}
    \label{eq:Minimax1}
    \inf_{x\in C}\sup_{y\in D}f(x,y)=\sup_{y\in D}\inf_{x\in C}f(x,y).
  \end{equation}

\end{theorem}

\begin{proof}
  Note first that ``$\geq$'' is always true whatever $C$, $D$ and $f$
  are. Thus there is nothing to prove if $\alpha:=\sup_{y\in
    D}\inf_{x\in C}f(x,y)=\infty$, hence we assume
  $\alpha<\infty$. 

  For any $y\in D$ and $\beta\in\RB$, we set $A_y^\beta:=\{x\in C:\,
  f(x,y)\leq \beta\}$. Then \citep[][Theorem~1]{MR764631} implies that
  the family $\{A^{\alpha+\varepsilon}_y\}_{y\in D}$ has the finite
  intersection property for every $\varepsilon>0$. Noting that each
  $A^{\alpha+\varepsilon}_y$ is compact by assumption made on $f$, we
  have $\bigcap_{y\in D}A_y^{\alpha+\varepsilon}\neq\emptyset$
  (indeed, fixing arbitrary $y_0\in D$, we have
  $A_{y_0}^{\alpha+\varepsilon}$ is compact,
  $A_y^{\alpha+\varepsilon}\cap A_{y_0}^{\alpha+\varepsilon}$ is its
  non-empty closed subset for each $y\in D$, and $\bigcap_{y\in
    D}A_y^{\alpha+\varepsilon}=\bigcap_{y\in
    D}(A_y^{\alpha+\varepsilon}\cap A_{y_0}^{\alpha+\varepsilon})\neq
  \emptyset$). But this is a necessary and sufficient condition for
  the equality (\ref{eq:Minimax1}) by \citep[][Theorem~2]{MR764631}.
\end{proof}

  \begin{proposition}
    \label{prop:Truncation}
    For a finite monotone convex function $\varphi$ with the Fatou
    property on a solid space $\Xs$ containing the constants, the
    Lebesgue property is equivalent to: for any countable net
    $(X_\alpha)_\alpha$,
    \begin{equation}
      % \label{eq:AproxBDD}
      \tag{\ref{eq:AproxBDD}}
      X_\alpha\in L^\infty,\, |X_\alpha|\leq |X|,\,\forall \alpha,\text{ and }\,X_\alpha\rightarrow X
      \text{ a.s. } \Rightarrow \, \varphi(X_\alpha)\rightarrow\varphi(X).
    \end{equation}
  \end{proposition}
  \begin{proof}
    The necessity is clear from
    Remark~\ref{rem:Lebesgue-OrderConti}. Recall that the Lebesgue
    property of $\varphi$ is equivalent to the sequential continuity
    from above.  For a sequence $(X_n)_n\subset \Xs$ with
    $X_n\downarrow X\in\Xs$, consider a net
    $X_{n,m}:=(X_n\vee(-n))\wedge m$ with indices $(n,m)$ directed by
    $(n,m)\preceq(n',m')$ iff $n\leq n'$ and $m\leq m'$. Then
    $X_{n,m}\in L^\infty$ for each $(n,m)$ and
    $X_{n,m}\stackrel{o}\rightarrow X$ in $\Xs$. Indeed,
    $\limsup_{(n,m)}X_{n,m}=\inf_{(n,m)}\sup_{n'\geq n, m'\geq
      m}(X_{n'}\vee (-n'))\wedge m'=\inf_{(n,m)} X_n\vee (-n)=X$, and
    $\liminf_{(n,m)}X_{n,m}=\sup_{(n,m)}\inf_{n'\geq n,m'\geq m}
    (X_{n'}\vee (-n'))\wedge m'=\sup_{(n,m)}X\wedge m=X$. Therefore
    $\varphi(X)=\lim_{(n,m)}\varphi(X_{n,m})$ by
    (\ref{eq:AproxBDD}). On the other hand,
    $\varphi(X_n)\leq\varphi(X_n\vee (-n)) =\sup_m\varphi((X_n\vee
    -n)\wedge m)$ by Fatou and monotonicity, thus
  \begin{align*}
    \inf_n\varphi(X_n)&\leq \inf_n\sup_m \varphi((X_n\vee -n)\wedge
    m)=\lim_n\lim_m\varphi((X_n\vee -n)\wedge m)\\
    &=\lim_{(n,m)}\varphi(X_{n,m})=\varphi(X).
  \end{align*}
  Hence $\varphi$ has the Lebesgue property.
\end{proof}
%\addtocontentsline{toc}{section}{reference}

% \bibliographystyle{OwariRev}
% \bibliography{main}
\end{document}